\documentclass[review]{elsarticle}
\usepackage{srcltx}
\usepackage{eurosym}
\usepackage{mathtools}
\usepackage{amsmath}
\usepackage{amsfonts}
\usepackage{amssymb}
\usepackage{amsthm}
\usepackage{graphicx}
\usepackage{mathrsfs}
\usepackage{xcolor}
\usepackage{exscale}
\usepackage{latexsym}
\usepackage{esvect}

\usepackage[colorlinks,plainpages=true,pdfpagelabels,hypertexnames=true,colorlinks=true,pdfstartview=FitV,linkcolor=blue,citecolor=red,urlcolor=black,backref=page]{hyperref}
\PassOptionsToPackage{unicode}{hyperref}
\PassOptionsToPackage{naturalnames}{hyperref}
\usepackage{enumerate}
\usepackage[shortlabels]{enumitem}
\usepackage{bookmark}
\usepackage{wasysym}
\usepackage{esint}
\usepackage[titletoc,toc,page]{appendix}
\usepackage[ddmmyyyy]{datetime}
\numberwithin{equation}{section}
\everymath{\displaystyle}
\usepackage[margin=2.2cm]{geometry}

\usepackage[capitalize,nameinlink]{cleveref}
\crefname{section}{Section}{Sections}
\crefname{subsection}{Subsection}{Subsections}
\crefname{condition}{Condition}{Conditions}
\crefname{hypothesis}{Hypothesis}{Hypothesis}
\crefname{assumption}{Assumption}{Assumptions}
\crefname{lemma}{Lemma}{Lemmas}
\crefname{claim}{Claim}{Claims}
\crefname{remark}{Remark}{Remarks}

\crefformat{equation}{\textup{#2(#1)#3}}
\crefrangeformat{equation}{\textup{#3(#1)#4--#5(#2)#6}}
\crefmultiformat{equation}{\textup{#2(#1)#3}}{ and \textup{#2(#1)#3}}
{, \textup{#2(#1)#3}}{, and \textup{#2(#1)#3}}
\crefrangemultiformat{equation}{\textup{#3(#1)#4--#5(#2)#6}}%
{ and \textup{#3(#1)#4--#5(#2)#6}}{, \textup{#3(#1)#4--#5(#2)#6}}%
{, and \textup{#3(#1)#4--#5(#2)#6}}

\Crefformat{equation}{#2Equation~\textup{(#1)}#3}
\Crefrangeformat{equation}{Equations~\textup{#3(#1)#4--#5(#2)#6}}
\Crefmultiformat{equation}{Equations~\textup{#2(#1)#3}}{ and \textup{#2(#1)#3}}
{, \textup{#2(#1)#3}}{, and \textup{#2(#1)#3}}
\Crefrangemultiformat{equation}{Equations~\textup{#3(#1)#4--#5(#2)#6}}%
{ and \textup{#3(#1)#4--#5(#2)#6}}{, \textup{#3(#1)#4--#5(#2)#6}}%
{, and \textup{#3(#1)#4--#5(#2)#6}}

\crefdefaultlabelformat{#2\textup{#1}#3}
\newtheorem{theorem}{Theorem}[section]
\newtheorem{lemma}[theorem]{Lemma}
\newtheorem{corollary}[theorem]{Corollary}
\newtheorem{claim}[theorem]{Claim}
\newtheorem{proposition}[theorem]{Proposition}

\newtheorem{definition}[theorem]{Definition}
\newtheorem{remark}[theorem]{Remark}        
  
\numberwithin{equation}{section}
\def\Yint#1{\mathchoice
	{\YYint\displaystyle\textstyle{#1}}%
	{\YYint\textstyle\scriptstyle{#1}}%
	{\YYint\scriptstyle\scriptscriptstyle{#1}}%
	{\YYint\scriptscriptstyle\scriptscriptstyle{#1}}%
	\!\iint}
\def\YYint#1#2#3{{\setbox0=\hbox{$#1{#2#3}{\iint}$}
		\vcenter{\hbox{$#2#3$}}\kern-.50\wd0}}
\def\longdash{-\mkern-9.5mu-} 
\def\tiltlongdash{\rotatebox[origin=c]{18}{$\longdash$}}
\def\fiint{\Yint\tiltlongdash}

\def\XXint#1#2#3{{\setbox0=\hbox{$#1{#2#3}{\int}$}
		\vcenter{\hbox{$#2#3$}}\kern-.50\wd0}}

\makeatletter
\def\namedlabel#1#2{\begingroup
	\def\@currentlabel{#2}%
	\label{#1}\endgroup
}
\makeatother
\makeatletter
\newcommand{\rmh}[1]{\mathpalette{\raisem@th{#1}}}
\newcommand{\raisem@th}[3]{\hspace*{-1pt}\raisebox{#1}{$#2#3$}}
\makeatother

\newcommand{\lsbo}[2]{#1_{\rmh{-1pt}{#2}}}

\newcommand{\redref}[2]{\texorpdfstring{\protect\hyperlink{#1}{\textcolor{black}{(}\textcolor{red}{#2}\textcolor{black}{)}}}{}}
\newcommand{\redlabel}[2]{\hypertarget{#1}{\textcolor{black}{(}\textcolor{red}{#2}\textcolor{black}{)}}}

\newcommand{\descitem}[2]{\item[(#1)]\label{#2}}
\newcommand{\descref}[2]{\hyperref[#1]{\textcolor{black}{(}\textcolor{blue}{\bf #2}\textcolor{black}{)}}}

\newcommand{\dref}[2]{\hyperref[#1]{\textcolor{black}{(}\textcolor{blue}{\bf #2}\textcolor{black}{)}}}
\makeatletter
\g@addto@macro\normalsize{%
	\setlength\abovedisplayskip{3pt}
	\setlength\belowdisplayskip{3pt}
	\setlength\abovedisplayshortskip{1pt}
	\setlength\belowdisplayshortskip{3pt}
}
\makeatletter
\def\ps@pprintTitle{%
	\let\@oddhead\@empty
	\let\@evenhead\@empty
	\def\@oddfoot{}%
	\let\@evenfoot\@oddfoot}
\makeatother
\usepackage{setspace}
\def\aa{\mathcal{A}}

\newcommand{\lamot}{\La_0,\La_1}

\newcommand{\tf}{\tilde{f}}

\newcommand{\tu}{\tilde{u}}
\newcommand{\tv}{\tilde{v}}

\newcommand{\mfx}{\mathfrak{x}}

\newcommand{\mfz}{\mathfrak{z}}

\newcommand{\mft}{\mathfrak{t}}

\newcommand\RR{\mathbb{R}}

\newcommand\NN{\mathbb{N}}
\newcommand{\al}{\alpha}

\newcommand{\ga}{\gamma}
\newcommand{\de}{\delta}
\newcommand{\ve}{\varepsilon}

\newcommand{\sig}{\sigma}

\newcommand{\la}{\lambda}

\newcommand{\Sig}{\Sigma}

\newcommand{\Om}{\Omega}

\newcommand{\La}{\Lambda}

\DeclareMathOperator{\dv}{div}
\DeclareMathOperator{\spt}{spt}

\DeclareMathOperator{\diam}{diam}

\DeclareMathOperator{\loc}{loc}

\newcommand{\iprod}[2]{\langle #1 \ ,  #2\rangle}

\newcommand{\abs}[1]{\left| #1\right|}

\newcommand{\lbr}[1][(]{\left#1}
\newcommand{\rbr}[1][)]{\right#1}

\newcommand{\avgs}[2]{\lsbo{\lbr #1 \rbr}{#2}}

\newcommand{\txt}[1]{\qquad \text{#1} \qquad}
\newcommand{\mfa}{\mathfrak{\bf a}}

\newcommand{\nonfe}{\mfa'\lbr 1-\frac{p}{\mfa}\rbr}
\newcommand{\na}{\nabla}
\DeclareMathOperator{\osc}{osc}

\newcommand{\Qfrho}{Q_{4\rho}^\la(z_0)}

\newcommand{\Qrho}{Q_{\rho}^\la(z_0)}

\newcommand{\rhoa}{\rho_a}
\newcommand{\rhob}{\rho_b}

\newcommand{\pa}{\partial}

\newcommand{\avgsueta}[1][]{\avgs{u}{\eta}^{#1}(t)}
\newcommand{\avgsub}{\avgs{u}{B_r}(t)}
\newcommand{\avgsutrho}{\avgs{u}{B_{2\rho}}(t)}

\newcommand{\rhoz}{\rho_{\mathfrak{z}}}
\newcommand{\bF}{\mathbf{F}}
\newcommand{\bH}{\mathbf{H}}

\newcommand{\bM}{\mathbf{M}}
\newcommand{\bP}{\mathbf{P}}

\newcommand{\eduj}{E_p(\na u,Q_{j})}
\newcommand{\edujj}{E_p(\na u,Q_{j+1})}
\newcommand{\edvj}{E_p(\na v_j,\tfrac{1}{2}Q_{j})}
\newcommand{\edvjj}{E_p(\na v_j,Q_{j+1})}
\let\oldosc\osc
\renewcommand{\osc}{\oldosc\limits}
\newcounter{whitney}
\refstepcounter{whitney}

\newcounter{ineqcounter}
\refstepcounter{ineqcounter}

\begin{document}
\begin{frontmatter}

\title{Borderline gradient regularity estimates for quasilinear parabolic systems with data independent of time}

\author[myaddress4]{Karthik Adimurthi\tnoteref{thanksfirstauthor}}
\ead{karthikaditi@gmail.com and kadimurthi@tifrbng.res.in}
\tnotetext[thanksfirstauthor]{Supported by the Department of Atomic Energy,  Government of India, under
	project no.  12-R\&D-TFR-5.01-0520 and  SERB grant SRG/2020/000081}

\author[myaddress]{Wontae Kim}
\ead{kim.wontae.pde@gmail.com and m20258@snu.ac.kr}

\address[myaddress4]{Tata Institute of Fundamental Research, Centre for Applicable Mathematics,Bangalore, Karnataka, 560065, India}
\address[myaddress]{Research Institute of Mathematics, Seoul National University, Seoul 08826, Korea.}

\begin{abstract}
In this paper, we study some regularity issues concerning the gradient of weak solutions of 
\[
    u_t - \dv \aa(x,t,\nabla u) = g,
\]
where $\aa(x,t,\nabla u)$ is modeled after the $p$-Laplace operator. The main results we are interested in is to obtain optimal conditions on the datum $g$ (independent of time) such that borderline higher integrability of the gradient and Lipschitz estimates for the weak solution holds.  Moreover, we develop a theory where we can obtain elliptic type estimates using parabolic theory, which gives improved  potential estimates for the elliptic systems.
\end{abstract}


\end{frontmatter}

\begin{singlespace}
\tableofcontents
\end{singlespace}

\section{Introduction}
\label{section1}

In this paper, we are interested in studying the gradient regularity of weak solutions of the system
\begin{equation}\label{main_eqn}
u_t - \dv (|\nabla u|^{p-2}\nabla u) = g(x,t),
\end{equation}
where the datum $|g(x,t)| \leq f(x)$ is independent of time. Our main goal is to obtain elliptic type estimates using parabolic techniques and as a goal, obtain essentially optimal endpoint conditions required for such regularity.

\subsection{Discussion about Gradient higher integrability}

The first part of the paper deals with the higher integrability of weak solutions for systems with data independent of time. In the case
\[
u_t - \dv |\nabla u|^{p-2}\nabla u = \dv |F(x,t)|^{p-2}F(x,t),
\]
we have the following implications: there exists $\ve_0$ such that for any $\ve \in (0,\ve_0)$, if $F \in L^{p+\ve}_{\loc}$, then we have $|\nabla u| \in L^{p+\ve}$. The natural question to ask is when $\ve \rightarrow 0$, how much higher integrability for the gradient of the solution can be retained? Such an end point result was obtained for elliptic systems in \cite{MR2122416}.

The first result we have from \cref{thm_1} states that if the data $g=f(x) \in L^{\mfa'}\log L$, then $|\nabla u| \in L^p \log L$, so in particular, we obtain a logarithmic improvement of integrability. 

Our approach is primarily inspired by the ideas developed in \cite{MR1749438} where they developed the intrinsic scaling techniques required to study higher integrability of the gradient of the solutions of \cref{main_eqn}. In our situation, we cannot directly apply the techniques from \cite{MR1749438} for two reasons, the first being that the datum is a function whose scaling relations are not entirely obvious and the second being, since the data is independent of time, we need to somehow obtain estimates for the parabolic system on every time slice as opposed to over space-time cylinders. 

We overcome the first obstacle by considering the intrinsic scaling of the form \cref{ge_la} and \cref{le_la} which captures the right relation between the datum $f(x)$ and $|\nabla u|$. To overcome the second difficulty, we make use of an idea from \cite{MR2468726}, where we use test functions depending on time as in \cref{caccio}, which is possible thanks to the cancellation from \cite[Lemma 5.4]{MR2342615}. An important advantage we gain from such a choice of test function is that we do not need any 'gluing type' lemma (see \cite[Lemma 3.1]{MR1749438} for an example of such a lemma) and as a consequence, our proof simplifies the calculations even when $g \equiv 0$.

In the case $g =g(x,t)$, we can write the datum $g$ in divergence form for almost every time slice using the Riesz potential of order one as in \cref{lem_1.34}. In this form, we can now directly apply the result from \cite{MR1749438} to obtain \cref{thm_2}. Applying the Riesz potential bounds, we obtain \cref{cor_2} which proves higher integrability of the gradient of the solution with an elliptic type condition for the datum in the space variable and the scaling deficit is required only in the time variable. 
\subsection{Gradient potential estimates}

In this section, our main interest is to obtain optimal potential bounds for the gradient of solutions to general nonlinear parabolic  systems of the form \cref{p_main} which in particular implies the corresponding result for the elliptic systems. We will recall the history of the problem for the elliptic system and refer to \cite{MR3273649,MR3174278,MR4021174} for more details regarding this problem. We also highlight the paper \cite{MR4377996} where boundary potential estimates were obtained for more general elliptic systems with $(p,q)$-growth.

The potential estimates for the gradient of the solution were obtained in the important paper \cite{MR2746772}. In particular, they showed that  the gradient of solutions can be estimated by linear potentials, noting that these are sharper than Wolff potential bounds when $p \geq 2$ (see also  \cite{MR2719282,MR3174278,MR3004772,MR4331020,MR4438896,MR4078712} for more in this direction).

Our main interest is to obtain improved potential estimates for elliptic systems in the case $p\geq n$ and in order to do this, we follow the parabolic approach developed in \cite{MR2746772}. As a consequence, our proof is considerably simpler and requires only \cref{diff,diff_n}, which is the main contribution of this section. In \cite{MR4021174}, they were able to obtain boundedness of the gradient of the solution of elliptic systems for the case $p=n=2$ under the assumption $f \in L^q$ with $ q \geq \tfrac32$, whereas we obtain logarithmic potential which is finite whenever $f \in L^{1+\ve}$ for any $\ve >0$. We refer to the introduction of \cite{MR4021174} for an overview of the historical development. 

In order to highlight the main idea of the proof, we only consider equations of the form \cref{p_main}, even though we can consider more general structures as in \cite[(1.3)-(1.5)]{MR4021174} and we leave these details to the interested reader. An important aspect of our proof is that we use parabolic techniques developed in \cite{MR3273649} to obtain the required result (see \cref{thm_3}), which is considerably simpler than those developed in \cite{MR4021174}. In this regard, we also mention the recent papers \cite{MR4331020,MR4438896} where some interesting potential estimates are obtained for nonautonomous functionals.

\section{Main Theorems}
\label{section4}
Before we state the main theorems, let us recall the scaling deficit exponent given by 
\begin{equation}
\label{def_d}
d:=\left\{
\begin{array}{ll}
\frac{p}{2} &  \text{if} \ p \geq 2,\\
\frac{2p}{p(n+2)-2n} &  \text{if} \ p \leq 2,\\
\end{array}\right.
\end{equation}

\begin{remark}
	We will use the function $g = g(x,t)$ and $f = f(x)$ to denote functions that depend on time and are independent of time respectively. 
\end{remark}

The first theorem we prove is a borderline higher integrability result in the case of the data independent of time.
\begin{theorem}
	\label{thm_1}
	Let $u\in C(0,T;L^2(\Omega))\cap L^p(0,T;W^{1,p}(\Om))$ be a weak solution of \cref{main_eqn} under the assumption \cref{str} and $|g(x,t)|\leq f(x)$. Furthermore, let $\de\in(0,1)$ and  $\mfa$ be as defined in \cref{reverse_exp} and \cref{mfa_high} respectively. Then for any $a>0$, the following borderline higher integrability result holds:
	\begin{equation*}
	\fiint_{Q_{r}(z_0)}|\na u|^p\log(e+|\na u|^a) \ dz
	\apprle_{(n,p,\La_0,\La_1,\de,a)}\log(e+\la_0^a)\fiint_{Q_{2r}(z_0)}(|\na u|+1)^p\ dz+\fint_{B_{2r}(x_0)}|\tf|\log(e+|\tf|)\ dz.
	\end{equation*}
	Here, we have set
	\begin{gather*}
	\la_0^\frac{p}{d}:=\fiint_{Q_{2r}(z_0)}\lbr|\na u|+1\rbr^p\ dz+\lbr\fint_{B_{2r}(x_0)}(2r)^{\mfa'}|f|^{\mfa'}\ dx\rbr^\alpha \txt{and}
	|\tf(x)|:= \la_0^{\nu} (2r)^{\mfa'} |f(x)|^{\mfa'},
	\end{gather*}
	where $d$ is defined in \cref{def_d} and $\alpha,\nu$ are as defined in \cref{definition_constants}.
	
\end{theorem}

\begin{corollary}\label{cor_1}
	Following \cref{thm_1}, given $m_0 \in \NN$ and any $a >0$, let us denote $$\log_{(m_0)}(t) := \underbrace{\log(e+t) \log(e+t) \cdots \log(e+t)}_{m_0 \text{ times}},$$ then we have the following borderline higher integrability estimates
	\begin{equation*}
		\begin{array}{rcl}
			\fiint_{Q_{r}(z_0)}|\na u|^p\log_{(m_0)}(e+|\na u|^a) \ dz
			&\apprle&\sum_{i=0}^{m_0}\log_{(i)}(e+\la_0^a)\fiint_{Q_{2r}(z_0)}(|\na u|+1)^p\ dz\\
			&&+\sum_{i=0}^{m_0}\fint_{B_{2r}(x_0)}|\tf|\log_{(i)}(e+|\tf|)\ dz, \label{est_cor_1}.
		\end{array}
	\end{equation*}
	
\end{corollary}

From \cref{div_riesz}, we see that a measurable function admits a divergence form representation and thus, we can directly use the higher integrability result from \cite{MR1749438} to conclude the following theorem:
\begin{theorem}
	\label{thm_2}
	Let $u\in C(0,T;L^2(\Omega))\cap L^p(0,T;W^{1,p}(\Om))$ be a weak solution of \cref{main_eqn} under the assumption \cref{str}. Then there exists $\ve_0=\ve_0(n,p,\La_0,\La_1)$ such that for any $\ve\in(0,\ve_0)$, there holds
	\begin{equation*}
	\fiint_{Q_r(z_0)}|\na u|^{p+\ve}\ dz\apprle_{(n,p,\La_0,\La_1)}\la_0^{\ve}\fiint_{Q_{2r}(z_0)}(|\na u|+1)^p\ dz+\fiint_{Q_{2r}(z_0)}|I_1g|^{\frac{p}{p-1}+\frac{\ve}{p-1}}\ dz.
	\end{equation*}
	Here, $I_1$ is a Reisz potential of order one (see \cref{def_rz}) and 
	\begin{equation*}
	\begin{array}{rcl}
	\la_0^\frac{p}{d}
	:=\fiint_{Q_{2r}(z_0)}(|\na u|+1)^p+|I_1g|^\frac{p}{p-1}\ dz,
	\end{array}
	\end{equation*}
	and $d$ is defined in \cref{def_d}.
\end{theorem}
Subsequently, applying the Riesz potential bounds from \cref{lem_1.34}, we have the following important corollary where we obtain an elliptic type condition in the space variable and the parabolic condition is required only on the time variable of the datum:
\begin{corollary}
	\label{cor_2}
	Let $\frac{2n}{n+2}<p\le n$ and $u\in C(0,T;L^2(\Omega))\cap L^p(0,T;W^{1,p}(\Om))$ be a weak solution of \cref{main_eqn} under the assumption \cref{str}. And for any $\de\in(0,1)$, let $\mfa$ be defined in \cref{mfa_high} and $g\in L^{p'}(0,T;L^{\mfa'}(\Om))$. Then there exists $\ve_0=\ve_0(n,p,\La_0,\La_1,\de)$ such that for any $\ve\in(0,\ve_0)$, there holds that
	\begin{equation*}
	\fiint_{Q_r(z_0)}|\na u|^{p+\ve}\ dz\apprle_{(n,p,\La_0,\La_1,\de)}\la_0^{ \ve}\fiint_{Q_{2r}(z_0)}(|\na u|+1)^p\ dz+\fint_{I_{2r}(t_0)}\lbr \fint_{B_{2r}(x_0)}(2r)^{\mfa'}|g|^{\mfa'}\ dx\rbr^{\frac{p}{(p-1)\mfa'}+\frac{\ve}{(p-1)\mfa'}}\ dt,
	\end{equation*}
	where $d$ is from \cref{def_d} and 
	\[
	\la_0^\frac{p}{d}
	:=\fiint_{Q_{2r}(z_0)}(|\na u|+1)^p\ dz + \fint_{I_{2r}(t_0)}\lbr \fint_{B_{2r}(x_0)}(2r)^{\mfa'}|g|^{\mfa'}\ dx\rbr^{\frac{p}{(p-1)\mfa'}}\ dt.
	\]
	
	In particular, if $|g(x,t)| \leq |f(x)|$, then we have 
	\begin{equation*}
	\fiint_{Q_r(z_0)}|\na u|^{p+\ve}\ dz\apprle_{(n,p,\La_0,\La_1,\de)}\la_0^{\ve}\fiint_{Q_{2r}(z_0)}(|\na u|+1)^p\ dz+\lbr\fint_{B_{2r}(x_0)}(2r)^{\mfa'}|f|^{\mfa'}\ dx\rbr^{\frac{p}{(p-1)\mfa'}+\frac{\ve}{(p-1)\mfa'}}.
	\end{equation*}
\end{corollary}

\begin{remark}
	In \cref{cor_2}, we note that the relation between $\ve_0$ and $\de$ is of the form $\ve < \frac{p(1-\de)}{p\de -1}$. So in particular, as $\de \rightarrow 0$, we automatically have $\ve \rightarrow 0$. 
\end{remark}

We now state the gradient potential estimate in the case of data independent of time.
\begin{theorem}
	\label{thm_3}
	Let $u\in C(0,T;L^2(\Omega))\cap L^p(0,T;W^{1,p}(\Om))$ be a weak solution of \cref{p_main} and  suppose that $\bF$ defined in \descref{C8}{C8} - \descref{C10}{C10} of \cref{def5.1} is finite. Furthermore let $z_0$ be a Lebesgue point of $\nabla u$ and assume $\bF$ is finite at $z_0$. Then there holds
	\begin{equation*}
	|\nabla u(z_0)| \apprle_{n,p}\lbr\fiint_{Q_{2\rho}(z_0)}|\na u|^p+1\ dz\rbr^\frac{d}{p}+\bF^{\max\{1,\frac{1}{p-1}\}}(z_0),
	\end{equation*}
	where $d$ is defined in \cref{def_d}.
\end{theorem}

\section{Notations and Definitions}
\label{section2}

\subsection{Structure of the operator and notion of solution}

\begin{enumerate}[(i)]
    \item For the results of \cref{section5}, we assume $\aa(x,t,\nabla u)$ is a Carath\'eodory function, i.e., we have $(x,t) \mapsto \aa(x,t,\zeta)$  is measurable for every $\zeta \in \RR^{n}$ and 
$\zeta \mapsto \aa(x,t,\zeta)$ is continuous for almost every  $(x,t) \in \Om_T$.

We further assume that for a.e. $(x,t) \in \Om_T$ and for any $\zeta \in \RR^{n}$, there exist two positive constants $\lamot$, such that  the following bounds are satisfied   by the nonlinear structure:
\begin{gather}\label{str}
\iprod{\aa(x,t,\zeta)}{\zeta} \geq \La_0 |\zeta |^{p}  \txt{and} |\aa(x,t,\zeta)| \leq \La_1 |\zeta|^{p-1}.
\end{gather}

\item For the gradient potential estimate proved in \cref{section7}, we take $\aa(x,t,\nabla u) = |\nabla u|^{p-2} \nabla u$.
\end{enumerate}

\begin{definition}\label{def_solution}
We say that $u \in C_{\loc}^0(I; L^2(B)) \cap L^p(I; W^{1,p}(B))$ is a weak solution of \cref{main_eqn} if 
\[
\iint_{Q} g \phi \ dz = -\iint_{Q} u \varphi_{t} \,dt + \iint_{Q} \iprod{\mathcal{A}(x,t,u,\nabla u)}{\nabla \varphi} \,dz,
\]
for all $\varphi \in C^{1}(\bar{Q})$ which vanish on the parabolic boundary (see \descref{notparbnd}{N10}) of $Q$.  

\end{definition}

\subsection{Orlicz Space}
Let us recall some well known properties of Orlicz function spaces, see \cite[Chapter 8]{MR0450957} for the details.

\begin{definition}[$N$-function]\label{Nfn}
We say that a function $\Phi:[0,\infty)\to[0,\infty)$ is an $N$-function if it is a convex function satisfying
\begin{gather*}
\lim_{s\to0^+}\frac{\Phi(s)}{s}=0\txt{and}\lim_{s\to\infty}\frac{\Phi(s)}{s}=\infty.
\end{gather*}
    
\end{definition}

\begin{definition}[$\Delta_2$ condition]\label{delta2}
        An $N$-function $\Phi$ is said to be $\Delta_2$-regular if there exists a constant $L>0$ such that $\Phi(2s)\le L\Phi(s)$ for all $s \geq s_0$ with a given $s_0$ large. 
\end{definition}

Subsequently, we can define the Orlicz function space as follows:
\begin{definition}\label{orlicz_fnc}
    Let $\Phi$ be an $N$-function satisfying $\Delta_2$-regular condition, then the following function space is well defined and is a Banach space:
\begin{gather*}
    L^{\Phi}(\Omega):=\left\{ v :  \int_{\Omega}\Phi(|v|)\ dx<\infty\right\}\txt{and}\lVert v\rVert_{L^\Phi(\Omega)}:=\inf_{s>0}\left\{ \int_{\Omega}\Phi\lbr\frac{|v|}{s}\rbr\ dx\le1\right\}.
\end{gather*}    
\end{definition}
If two $N$-functions $\Phi_1$ and $\Phi_2$ satify $0<\lim_{s\to\infty}\frac{\Phi_1(s)}{\Phi_2(s)}<\infty$, then $L^{\Phi_1}(\Omega)=L^{\Phi_2}(\Omega)$ with norm equivalence. 
For a given $N$-function, we define the conjugate $N$-function of $\Phi$, denoted by $\tilde{\Phi}$, as
\begin{equation*}
\tilde{\Phi}(s) =\max_{t\ge0}(st-\Phi(t)).
\end{equation*}
Subsequently, the  following generalized H\"older inequality holds (see \cite[Chapter 8]{MR0450957} for the details):
\begin{lemma}\label{orl_holder}
    For any $v\in L^{\Phi}(\Omega)$ and $w\in L^{\tilde{\Phi}}(\Omega)$, we have
    \begin{equation*}
\left|\int_{\Omega}vw\ dx\right|\le2\lVert v\rVert_{L^{\Phi}(\Omega)}\lVert w\rVert_{L^{\tilde{\Phi}}(\Omega)}.
\end{equation*}
\end{lemma}
Let us take $\Phi_1(s):=s\log(1+s)$ and $\Phi_2(s):=(1+s)\log(1+s)-s$, then it is easy to see that 
\begin{equation*}
\lim_{s\to\infty}\frac{\Phi_1(s)}{\Phi_2(s)}=1\txt{and}\tilde{\Phi}_2=\exp(s)-s-1.
\end{equation*}
Applying \cite[Lemma 8.6]{MR1730563}, we have
\begin{equation}\label{llog_form}
\lVert h\rVert_{L^{\Phi_1}(\Om)}\approx \int_{\Omega}|h|\log\lbr e+\frac{|h|}{\lVert h\rVert_{L^1(\Om)}}\rbr\ dx.
\end{equation}
Moreover, we trivialy have
\begin{equation}\label{l1_llog}
\int_{\Omega}|h|\ dx\le \int_{\Omega}|h|\log\lbr e+\frac{|h|}{\lVert h\rVert_{L^1(\Om)}}\rbr\ dx.
\end{equation}

\subsection{Riesz potential}
In this subsection, let us recall some well known properties of Riesz potential defined as follows:
\begin{definition}\label{def_rz}
	In $\RR^n$, the Riesz kernel denoted by $I_\alpha$ for some $0<\alpha<n$, is defined by
		$I_{\alpha}(x):=|x|^{\alpha-n}$.
	The Riesz potential of any measurable function $h$ is defined as follows:
	\begin{equation*}
		I_\alpha h:=I_{\alpha}*h:=\int_{\RR^n}\frac{h(y)}{|x-y|^{n-\alpha}}\ dy.
	\end{equation*}
\end{definition}
The following result on the boundedness of Riesz potential can be found in \cite[Lemma 1.34]{MR1461542}.
\begin{lemma}\label{lem_1.34}
	Let $1\le q< n$ and $\Om\subset\RR^n$ be a domain with finite measure. Then for $q\le s<q^*:=\frac{nq}{n-q}$, $\frac{1}{r}:=1-\lbr \frac{1}{q}-\frac{1}{s}\rbr$ and $\gamma:=\frac{(1-n)r}{n}+1$, there holds
	\begin{equation*}
		\lVert I_1h\rVert_{L^s(\Omega)}\le C^\frac{1}{r}\lVert h\rVert_{L^q(\Om)},
	\end{equation*}
	where $C=\frac{w_n}{\gamma}\lbr\frac{|\Om|}{w_n}\rbr^\gamma$ and $w_n$ is a volume of unit ball in $\RR^n$.
\end{lemma}
The main result regarding Riesz potentials that will be used in this paper is the following representation for writing any measurable function in divergence form.
\begin{lemma}\label{div_riesz}
	Let $\Om\subset \RR^n$ be a bounded domain with finite measure and $h\in L^s(\Omega)$. We extend $h$ to be zero outside $\Omega$, then there exists a vector field $\vec{h}\in L^s(\RR^n,\RR^n)$ such that $h=\dv \vec{h}$ satisfying the point-wise estimate 
	\begin{equation*}
		|\vec{h}(x)|\apprle_{(n)} I_1h(x).
	\end{equation*}
\end{lemma}

\begin{proof}
	First we extend $h$ to be zero outside $\Om$ and  consider the Green's function $G(x,y)$ associated to $-\Delta$ on the ball in $ \RR^n$.   Now  consider at the following problem:
	\begin{equation*}
		-\Delta \phi=h \qquad \text{ in }\RR^n.
	\end{equation*}
	From standard theory, we can then explicitly write the solution as $\phi=G*h=G*(-\Delta\phi)$ from which it is easy to see that the following holds:
	\begin{equation*}
		\Delta\phi(x)=\int_{\RR^n}\Delta_x G(x,y)h(y)\ dy=\int_{\RR^n}\Delta_x G(x,y)(-\Delta \phi(y))\ dy.
	\end{equation*}
	Therefore, we can write  $h(x)=-\dv_x(\na G*h)=:-\dv_x\vec{h}$, where $\na G*h$ is applied component wise. Since the green function $G(x,y)$ we consider was on $\RR^n$, we know that
		$|\na_x G(x,y)|\apprle_{(n)}\frac{1}{|x-y|^{n-1}}$ from which we obtain the desired pointwise estimate.
\end{proof}

\section{Some well known lemmas}
\label{section3}
In this section, let us recall some well known lemmas, the first of which is the well known Sobolev-Poincare inequality (for example, see \cite[Chapter I]{MR1230384} for the details).
\begin{lemma}
	\label{g_n}
	For some $\rho >0$, let $B_\rho(x_0)\subset \mathbb{R}^n$. For constants $\sig,q,r\in[1,\infty)$ and $\vartheta\in(0,1)$ such that $-\frac{n}{\sig}\leq \vartheta(1-\frac{n}{q})-(1-\vartheta)\frac{n}{r}$ is satisfied,  then for any $h\in W^{1,q}(B_\rho(x_0))$ there holds
	$$
	\fint_{B_\rho(x_0)}\frac{\left| h\right|^\sig}{\rho^\sig}\ dx\apprle_{(n,q)}\lbr\fint_{B_\rho(x_0)}\frac{\left| h\right|^q}{\rho^q}+|\nabla h|^q\ dx\rbr^\frac{\vartheta\sig}{q}\lbr\fint_{B_\rho(x_0)}\frac{\left| h\right|^r}{\rho^r}\ dx\rbr^\frac{(1-\vartheta)\sig}{r}.
	$$
\end{lemma}

The next result we need the following standard iteration lemma (for example, see \cite[Lemma 6.1]{MR1962933} or \cite{MR3273649} for the details):
\begin{lemma}
	\label{iter_lemma}
	Let $0< r< R<\infty$ be given and $h : [r,R] \to \RR$ be a non-negative and bounded function. Furthermore, let $\theta \in (0,1)$ and $A,B,\gamma_1,\gamma_2 \geq 0$ be fixed constants and 
	suppose that
	$$
	h(\rho_1) \leq \theta h(\rho_2) + \frac{A}{(\rho_2-\rho_1)^{\gamma_1}} + \frac{B}{(\rho_2-\rho_1)^{\gamma_2}},
	$$
	holds for all $r \leq \rho_1 < \rho_2 \leq R$, then the following conclusion holds:
	$$
	h(r) \apprle_{(\theta,\gamma_1,\gamma_2)} \frac{A}{(R-r)^{\gamma_1}} + \frac{B}{(R-r)^{\gamma_2}}.
	$$
\end{lemma}

Since we are studying structures with $p$-growth, we also need the following algebraic identities (see \cite[Chapter 8]{MR1962933} for the details).
\begin{lemma}\label{p_str}
	Let $p\in(1,\infty)$, then for any two vectors $A,B\in\RR^n$ (both non zero), the following holds:
	\begin{equation*}
	\left||A|^\frac{p-2}{2}A-|B|^\frac{p-2}{2}B\right|^2\apprle_{(n,p)} \lbr|A|^2+|B|^2\rbr^\frac{p-2}{2}|A-B|^2 \apprle_{(n,p)}\iprod{|A|^{p-2}A-|B|^{p-2}B}{A-B}.
	\end{equation*}
	Moreover, in the case when $p\ge2$, we have
	\begin{equation*}
	|A-B|^{p}\apprle_{(n,p)}\lbr|A|^2+|B|^2\rbr^\frac{p-2}{2}|A-B|^2,
	\end{equation*}
	and in the case when $p\le2$, there holds that
	\begin{equation*}
	|A-B|^p\apprle_{(n,p)}\left||A|^\frac{p-2}{2}A-|B|^\frac{p-2}{2}B\right|^2+\left||A|^\frac{p-2}{2}A-|B|^\frac{p-2}{2}B\right|^p|A|^\frac{(2-p)p}{2}.
	\end{equation*}
\end{lemma}
The following lemma is an elementary extension theorem and can be found in \cite[Theorem 1.63]{MR1461542}.
\begin{lemma}\label{ext}
	Let $1\le p<\infty$ and $h\in W^{1,p}(B_r(x_0))$. Then there exists $\tilde{h}\in W^{1,p}(\RR^n)$ such that $\tilde{h}=h$ on $B_r(x_0)$, $\spt \tilde{h}\subset B_{\frac{3}{2}r}(x_0)$ and 
	\begin{equation*}
		\int_{B_{\frac{3}{2}r}(x_0)}|\na \tilde{h}|^p\ dx\apprle_{(n,p)}\int_{B_r(x_0)}|\na h|^p+r^{-p}|h|^p\ dx.
	\end{equation*}
\end{lemma}
Let us now recall the well known Sobolev embedding result which can be found in \cite[Lemma 1.64]{MR1461542}.
\begin{lemma}
    \label{sob_poin}
    Let $B_r$ be a ball with radius $r$ and let $1 \leq p \leq n$ and $1\leq q \leq \frac{np}{n-p}$ with $1 \leq q < \infty$, then we have
    \[
        \lbr \fint_{B_r} \abs{\frac{h - \avgs{h}{B_r}}{r}}^{q} \ dx\rbr^{\frac{1}{q}} \leq C(n,p) \lbr \fint_{B_r} |\nabla h|^p \ dx \rbr^{\frac{1}{p}}.
    \]
    Here, we have taken $q \in (1,\infty)$ is any number in the case $p =n$.
\end{lemma}

We will use the following Morrey embedding and can be found in \cite[Theorem 1.62]{MR1461542}.
\begin{lemma}\label{Morrey}
	Let $h\in W_0^{1,p}(\Omega)$ with $p>n$. Then $h\in C^{0,1-\frac{n}{p}}(\overline{\Omega})$ and
	\begin{gather*}
	\sup_{x,y\in\overline{\Omega}}|h(x)-h(y)|\apprle_{n,p}|x-y|^{1-\frac{n}{p}}\txt{and}\sup_{\Om}|h|\apprle_{(n,p)}|\Omega|^\frac{1}{n}\lbr\fint_{\Omega}|\na h|^p\ dx\rbr^\frac{1}{p}.
	\end{gather*}
\end{lemma}
In the borderline case $p=n$, we have the following result, see \cite[Theorem 1.66]{MR1461542} for the details.
\begin{theorem}\label{J_N}
	Let $h\in W^{1,1}(\Omega)$ where $\Omega\subset \RR^n$ is convex. Suppose there is a constant $M$ such that 
	\begin{equation*}
	\int_{\Omega\cap B_r(x_0)}|\na h|\ dx\le Mr^{n-1},
	\end{equation*}
	for all $B_r(x_0)\subset \RR^n$,  then there exist positive constants $\sig_0=\sig_0(n)$ and $C=C(n)$ such that 
	\begin{equation*}
	\int_{\Omega}\exp\lbr\frac{\sig}{M}|h-(h)_\Omega|\rbr\ dx\le C(\diam\Omega)^n,
	\end{equation*}
	where $\sig< \frac{\sig_0|\Omega|}{(\diam\Omega)^{n}}$.
\end{theorem}

We define the  truncated function in space variable as follows:
\begin{definition}
	Let  $\eta\in C_c^\infty(B_r(x_0))$ be any cut-off function, then we denote
	\begin{equation*}
		\avgsueta[r]:=\frac{\int_{B_{r}}\eta^p(x)u(x,t)\ dx}{\int_{B_r}\eta^p(x)\ dx}.
	\end{equation*}
Note that we will keep track of the domain of the definition of the cut-off function $\eta$. 
\end{definition}
The proof of the following lemma can be found in \cite[Lemma 5.3]{MR2342615}.
\begin{lemma}\label{weight_lem}
	Let $B_{r}(x_0)\subset\Omega$ and  $u(\cdot,t)\in L^p_{\loc}(\Omega)$ for some  $p>1$ be given. Furthermore, let $\eta\in C_c^\infty(B_r(x_0))$ be a cut-off function such that the following is satisfied for some $\overline{c}>0$:
	\begin{equation*}
		\sup_{x\in B_r(x_0)}\eta(x)\le \overline{c}\fint_{B_r(x_0)}\eta(x)\ dx,
	\end{equation*}
	then the following conclusion holds:
	\begin{equation*}
		\int_{B_r}|u-\avgsub|^p\ dx\apprle_{(n,p,\overline{c})}\int_{B_r}|u-\avgsueta[r]|^p\ dx\apprle_{(n,p,\overline{c})}\int_{B_r}|u-\avgsub|^p\ dx.
	\end{equation*}
\end{lemma}

\subsection{Notation}
We shall clarify all the notation that will be used in this paper.

\begin{description}
	\descitem{N1}{not0}  We shall fix a point $z_0 = (x_0,t_0) \in \Om \times (0,T)=\Om_T$. 
	\descitem{N2}{not1} We shall use $\nabla$ to denote derivatives with respect to the space variable $x$.
	\descitem{N3}{not2} We shall sometimes alternate between using $\dot{h}$, $\pa_t h$ and $h'$ to denote the time derivative of a function $h$.
	\descitem{N4}{not3} We shall use $D$ to denote the derivative with respect to both the space variable $x$ and time variable $t$ in $\RR^{n+1}$.
	\descitem{N5}{not3-1} In what follows, we shall always assume the following bound holds:
	\begin{equation*}
	\label{bound_pp}
	\frac{2n}{n+2} < p < \infty.
	\end{equation*}
	
	\descitem{N6}{not4}  Let $z_0 = (x_0,t_0) \in \RR^{n+1}$ be a point, $\rho, s >0$ be two given parameters and let $\la \in [1,\infty)$. We shall use the following notations:
	\begin{equation*}\label{notation_space_time}
	\def\arraystretch{1.5}
	\begin{array}{lll}
	& I_s(t_0) := (t_0 - s^2, t_0+s^2) \subset \RR,
	& \quad Q_{\rho,s}(z_0) := B_{\rho}(x_0) \times I_{s}(t_0) \subset \RR^{n+1},\\ 
	& I_s^{\la}(t_0) := (t_0 - \la^{2-p}s^2, t_0+\la^{2-p}s^2) \subset \RR,
	&\quad Q^{\la}_{\rho,s}(z_0) := B_{\rho}(x_0) \times I^{\la}_{s}(t_0) \subset \RR^{n+1},\\
	&  Q_{\rho}^{\la} (z_0) := Q_{\rho, \rho}^{\la} (z_0). 
	\end{array}
	\end{equation*}
	
	\descitem{N7}{not5} We shall use $\int$ to denote the integral with respect to either space variable or time variable and use $\iint$ to denote the integral with respect to both space and time variables simultaneously. 
	
	Analogously, we will use $\fint$ and $\fiint$ to denote the average integrals as defined below: for any set $A \times B \subset \RR^n \times \RR$, we define
	\begin{gather*}
	\avgs{h}{A}:= \fint_A h(x) \ dx = \frac{1}{|A|} \int_A h(x) \ dx,\\
	\avgs{h}{A\times B}:=\fiint_{A\times B} h(x,t) \ dx \ dt = \frac{1}{|A\times B|} \iint_{A\times B} h(x,t) \ dx \ dt.
	\end{gather*}
	
	\descitem{N8}{not6} Given any positive function $\mu$, we shall denote $\avgs{h}{\mu} := \int h\frac{\mu}{\|\mu\|_{L^1}}dm$ where the domain of integration is the domain of definition of $\mu$ and $dm$ denotes the associated measure. 
	
	Moreover, we shall denote $\avgs{h}{\mu}^r := \int_{B_r} h\frac{\mu}{\|\mu\|_{L^1}}dm$,  where the superscript $r$ is used to keep track of the domain of integration. 
	
	\descitem{N9}{not11} We will use the notation $\apprle_{(a,b,\ldots)}$ to denote an inequality with a constant depending on $a,b,\ldots$.
	\descitem{N10}{notparbnd} For a given space-time cylinder $Q = B_R \times (a,b)$, we denote the parabolic boundary of $Q$ to be the union of the bottom and the lateral boundaries, i.e., $\pa_pQ = B_r \times \{t=a\} \bigcup \pa B_R \times (a,b)$.
	
	\descitem{N11}{not12} The exponent $\mfa$ denotes a fixed value that will be given explicitly at the beginning of each section. 
	
	\descitem{N12}{not13} We will denote $p' = \frac{p}{p-1}$ to be the conjugate exponent, $p^*= \frac{np}{n-p}$ to be the elliptic Sobolev exponent and $p^{\#} = \frac{p(n+2)}{n}$ to be the prabolic Sobolev exponent. 
\end{description}

\section{Proof of \texorpdfstring{\cref{thm_1}}. - Borderline higher integrability.}
\label{section5}
The proof proceeds in several steps, essentially following the strategy developed in \cite{MR1749438}. Before we begin the proof, let us fix some constants that will be needed later. 
\begin{definition}\label{def_delta}
Let us define the Sobolev exponent as follows:
	\begin{equation}\label{mfa_high}
	\mfa:=
	\begin{cases}
	(\de p)^*:=\frac{n \de p}{n- \de p}&\text{ if } p \leq  n,\\
	\infty&\text{ if } p> n.
	\end{cases}
	\end{equation}
\end{definition}
We also need the following choice of an exponent that will be needed when proving a reverse H\"older type inequality:
\begin{definition}\label{reverse_exp}
	For some $\de \in (0,1)$, we take
	\begin{equation*}
	\label{def_q}
		q:=\max\left\{\de p, \frac{np}{n+2},\frac{2n}{n+2} \right\},
	\end{equation*}
satisfying the following restrictions:
\begin{equation}
    \label{def_de}
    \de \geq \frac{n}{n+2}\max\left\{ 1, \frac{2}{p} \right\}, \qquad \de > \frac{1}{p} \ \ \text{in the case} \ p\leq n \txt{and} \de > \frac{n}{p} \ \ \text{in the case} \ p>n.
\end{equation}
Note that since $p > \frac{2n}{n+2}$, the above restrictions imply we can take $\de <1$.
	\end{definition}

With these choices, we shall now define the required intrinsic setting:
\begin{definition}
	Let $Q_{8\rho}^\la(z_0)\subset\Om_T$ for some $\la\ge1$. Furthermore, let us assume the following intrinsic bounds are satisfied throughout this section:
	    \begin{gather}
	    	\fiint_{\Qrho}\lbr|\na u|+1\rbr^p\ dz+\la^{\nonfe}\fint_{B_{\rho}(x_0)}\rho^{\mfa'}|f|^{\mfa'}\ dz\ge \la^p.\label{ge_la},\\
	\fiint_{\Qfrho}\lbr|\na u|+1\rbr^p \ dz+\la^{\nonfe}\fint_{B_{4\rho}(x_0)}(4\rho)^{\mfa'}|f|^{\mfa'}\ dz\le\la^p, \label{le_la}.
	\end{gather}
	
\end{definition}

\subsection{Caccioppoli type estimate}
Suppose that $B_{4\rho}(x_0)\subset\Om$ and $\rho\le \rhoa<\rhob\le4\rho$. Consider the following cut-off function:
\begin{equation}
\label{cut_off_function1}
\begin{array}{cccc}
\eta=\eta(x) \in C_c^{\infty}\lbr B_{\rho_b}(x_0)\rbr, & \eta \equiv 1 \ \, \text{on} \ B_{\rho_a}(x_0), & 0 \leq \eta \leq 1, & |\nabla \eta| \leq \frac{c}{\rho_b-\rho_a}.
\end{array}
\end{equation}
Then from the restriction $\rho\le \rhoa<\rhob\le4\rho$, it is clear that 
\begin{equation*}
\sup_{x\in B_{\rhob}(x_0)}\eta(x)\apprle_{(n)}\fint_{B_{\rhob}(x_0)}\eta(x)\ dx.
\end{equation*}
\begin{lemma}\label{caccio}
	Let $u$ be a weak solution of \cref{main_eqn} and  $B_{4\rho}(x_0)\subset\Om$. Then for $\rho\le \rhoa<\rhob\le4\rho$ and $\eta$ defined in \cref{cut_off_function1}, the following holds:
\begin{equation*}
\begin{array}{l}
\la^{p-2}\sup_{I_{\rhob}^\la(t_0)}\fint_{B_{\rhob}(x_0)}\abs{\frac{u-\avgsueta[\rhob]}{\rhob}}^2\eta^p\zeta^2\ dx+\fiint_{Q_{\rhob}^\la(z_0)}|\na u|^p\eta^p\zeta^2\ dz\\
\hspace*{3cm}\apprle_{(n,p,\La_0,\La_1)}\fiint_{Q_{\rhob}^\la(z_0)}\abs{\frac{u-\avgsueta[\rhob]}{\rhob-\rhoa}}^p\ dz+\la^{p-2}\fiint_{Q_{\rhob}^\la(z_0)}\abs{\frac{u-\avgsueta[\rhob]}{\rhob-\rhoa}}^2\eta^p\zeta\ dz\\
\hspace*{3cm}\qquad\qquad+\fiint_{Q_{\rhob}^\la(z_0)}|f||u-\avgsueta[\rhob]|\ dz.
\end{array}
\end{equation*}
\end{lemma}
\begin{proof}
    Consider the following cut-off functions:
	\begin{equation*}
		\begin{array}{cccc}
		\zeta=\zeta(t) \in C_c^{\infty}\lbr I_{\rhob}^{\la}(t_0)\rbr, & \zeta \equiv 1 \ \, \text{on} \ I_{\rhoa}^{\la}(t_0), & 0 \leq \zeta \leq 1, & |\pa_t \zeta| \leq \frac{c}{\la^{2-p}(\rhob-\rhoa)^2}.
		\end{array}
	\end{equation*}
	Let us make use of  $\lbr u-\avgsueta[\rho_b]\rbr\eta^p\zeta^2$ as a test function in  \cref{main_eqn}. Then the proof of Caccioppoli inequality follows verbatim as in  \cite[Lemma 5.4]{MR2342615}. There is one crucial cancellation that was used in \cite[Lemma 5.4]{MR2342615} and to highlight its importance, we recall that here. Using the test function in the time derivative, we see that
	\[
	\begin{array}{rcl}
	    \iint_{Q_{\rhob}}u \frac{d}{dt}\lbr[[]\lbr u-\avgsueta[\rho_b]\rbr\eta^p\zeta^2\rbr[]] &  = &  \iint_{Q_{\rhob}}\lbr u - \avgsueta[\rho_b]\rbr \frac{d}{dt}\lbr[[]\lbr u-\avgsueta[\rho_b]\rbr\eta^p\zeta^2\rbr[]]  \\
	    && + \iint_{Q_{\rhob}}\avgsueta[\rho_b] \frac{d}{dt}\lbr[[]\lbr u-\avgsueta[\rho_b]\rbr\eta^p\zeta^2\rbr[]],
	    \end{array}
	\]
	which requires us to estimate the last term which would not exist if $\avgsueta$ was independent of time. But the crucial cancellation that we need provides that this term is zero, which can be  formally seen as follows:
	\[
	    \int_{I_{\rhob}^\la(t_0)} \zeta^2(t)\underbrace{\lbr \int_{B_{\rhob}(x_0)} u(x,t) \eta^p(x) \ dx  - \frac{\int_{B_{\rhob}(x_0)} \eta^p(x) u(x,t) \ dx}{\int_{B_{\rhob}(x_0)} \eta^p(x) \ dx} \int_{B_{\rhob}(x_0)} \eta^p(x) \ dx \rbr}_{=0} \frac{d\avgsueta[\rho_b]}{dt}  \ dt = 0.
	\]
Once we have this cancellation, the rest of the Caccioppoli estimate follows verbatim as in \cite{MR1749438,MR2491806,MR2468726}.
\end{proof}

\begin{remark}
    \label{replace_eta}
    We will use the notation $\avgsueta[\rho_b]$ with the superscript $\rho_b$ to denote that the cut-off function $\eta$ is supported on $B_{\rhob}$ since we need to keep track of this radius in several forthcoming estimates.
\end{remark}

\subsection{Some intermediate estimates}
We need the following important estimate that will make use of Sobolev embedding:
\begin{lemma}\label{sobolev_bound}
    Recall the choice of $\mfa$ from \cref{mfa_high} and suppose \cref{le_la} holds. Then  for any $\rho \leq r \leq 4\rho$, the following holds:
    \begin{gather*}
         \fint_{I_r^{\la}(t_0)}\lbr \fint_{B_r(x_0)} \abs{\frac{u - \avgsueta[r]}{r}}^{\mfa} \ dx \rbr^{\frac{1}{\mfa}} dt  \apprle \la \txt{and}
         \fint_{I_r^{\la}(t_0)}\lbr \fint_{B_r(x_0)} \abs{\frac{u - \avgsueta[r]}{r}}^{2} \ dx \rbr^{\frac{1}{2}} dt \apprle \la.
    \end{gather*}
Note that when $\mfa =\infty$, we use the notation $\lbr\fint |f|^{\infty} \rbr^{\frac{1}{\infty}} := \| f\|_{\infty}$.
\end{lemma}
\begin{proof}
    Let us  prove the first estimate. We shall split the proof into two cases depending on the choice of $\mfa$ as given in \cref{mfa_high}. 
    \begin{description}
        \item[Case $p \leq n$:] In this case, we can apply Sobolev Poincare inequality in the space variable to get
        \[
        \fint_{I_r^{\la}(t_0)}\lbr \fint_{B_r(x_0)} \abs{\frac{u - \avgsueta[r]}{r}}^{\mfa} \ dx \rbr^{\frac{1}{\mfa}}\ dt 
         \overset{\text{\cref{sob_poin}}}{\apprle}  \lbr \fint_{I_r^{\la}(t_0)}\lbr \fint_{B_r(x_0)} |\nabla u|^{\de p} \ dx \rbr^{\frac{1}{\de p}} \ dt \rbr
         \overset{\cref{le_la,def_de}}{\apprle}  \la.
        \]

        \item[Case $p > n$:] In this case, note that $\mfa' = 1$. Let us first apply \cref{Morrey} to extend $u - \avgsueta[r]$ by a function $v$ on $B_{2r}(x_0)$. Thus we get the following sequence of estimates:
        \begin{equation*}
        \begin{array}{rcl}
            \sup_{B_r(x_0)} \abs{u - \avgsueta[r]} \overset{\text{\cref{ext}}}{\leq} \sup_{B_{2r}(x_0)} |v| & \overset{\text{\cref{Morrey}}}{\apprle}&  r\lbr \fint_{B_{2r}(x_0)} |\nabla v|^p \ dx \rbr^{\frac{1}{p}} \\
            & \overset{\text{\cref{ext}}}{\apprle} &  r \lbr \fint_{B_r(x_0)} |\nabla u|^p + \frac{\abs{u - \avgsueta[r]}^p}{r^p} \ dx \rbr^{\frac{1}{p}} \\
            & \overset{\redlabel{pna}{a}}{\apprle} & r \lbr \fint_{B_r(x_0)} |\nabla u|^p \ dx \rbr^{\frac{1}{p}},
        \end{array}
        \end{equation*}
        where to obtain \redref{pna}{a}, we applied Poincare's inequality. Thus, we get
        \begin{equation}\label{infity_estimate}
            \fint_{I_r^{\la}(t_0)} \sup_{B_r(x_0)} \frac{\abs{u - \avgsueta[r]}}{r} \ dt \apprle  \lbr \fiint_{Q_r^{\la}(z_0)} |\nabla u|^p \ dz\rbr^{\frac{1}{p}} \overset{\cref{le_la}}{\apprle} \la.
        \end{equation}

    \end{description}
    
    For the second estimate, we note that since $p > \frac{2n}{n+2}$, we have the Sobolev exponent $p^* = \frac{np}{n-p} > 2$ and thus, we can apply the standard Sobolev Poincare inequality to get the required estimate.
\end{proof}

We now prove an upper bound for the time term from \cref{caccio}.
\begin{lemma}\label{sup_2_bnd}
    Assuming \cref{caccio} holds, then for any $2\rho \leq r \leq 4\rho$, we have
    \begin{equation*}
		\sup_{I_{r}^\la(t_0)}\fint_{B_{r}(x_0)}\abs{\frac{u-\avgs{u}{\eta}^r(t)}{r}}^2 \eta^p \zeta^2\ dx\apprle \la^2,
	\end{equation*}
	where $\eta \in C_c^{\infty}(B_r(x_0)), \zeta \in C^{\infty}_c(I_{r}^{\la})(t_0)$ with $\eta \equiv 1$ on $B_s(x_0)$ and $\zeta \equiv 1$ on $I_s^{\la}(t_0)$ for some $\rho \leq s \leq r-\frac{\rho}{2}$.
\end{lemma}
\begin{proof}
Let $\rho \leq \rhoa \leq \frac32\rho$ and $2\rho \leq \rhob \leq 4\rho$ by any two radii, which implies $ \frac12\rho \leq \rhob-\rhoa\leq 3\rho$. In the notation of the lemma, this is the choice $s=\rhoa$ and $r = \rhob$ and $\eta \in C_c^{\infty}(B_{\rhob}(x_0))$ with $|\nabla \eta| \leq \frac1{\rhob-\rhoa}$ and  $\zeta \in C^{\infty}(I_{\rhob}^{\la}(t_0))$ with  $|\pa_t \zeta| \leq \frac{1}{\la^{2-p}(\rhob-\rhoa)^2}$. Let us make use of \cref{caccio} and estimate each of the terms on the right-hand side of \cref{caccio} to get:
\begin{description}
    \item[Estimate for the first term:] In order to estimate this, we proceed as follows:
    \begin{equation*}
        \begin{array}{rcl}
            \fiint_{Q_{\rhob}^\la(z_0)}\abs{\frac{u-\avgsueta[\rho_b]}{\rhob-\rhoa}}^p\ dz \overset{\text{\cref{sob_poin}}}{\apprle} \lbr \frac{\rhob}{\rhob-\rhoa}\rbr^p \fiint_{Q_{\rhob}^{\la}(z_0)} |\nabla u|^p \ dz\overset{\cref{le_la}}{\apprle}  \lbr \frac{\rhob}{\rhob-\rhoa}\rbr^p \la^p \apprle \la^p,
        \end{array}
    \end{equation*}
    where to obtain the last estimate, we made use of the restrictions on $\rhoa$ and $\rhob$. 
   \item[Estimate for the second term:]  We proceed as follows:
\begin{equation*}
				\begin{array}{rcl}
				II
				&=&\la^{p-2}\fint_{I_{\rhob}^\la(t_0)}\fint_{B_{\rhob}(x_0)}\abs{\frac{u-\avgsueta[\rho_b]}{\rhob-\rhoa}}\eta^\frac{p}{2}\zeta \abs{\frac{u-\avgsueta[\rho_b]}{\rhob-\rhoa}}\eta^\frac{p}{2}\zeta \ dx\ dt\\
				&\leq&\la^{p-2}\fint_{I_{\rhob}^\la(t_0)}\lbr\fint_{B_{\rhob}(x_0)}\abs{\frac{u-\avgsueta[\rho_b]}{\rhob-\rhoa}}^2\eta^p\zeta^2\ dx\rbr^\frac{1}{2}\lbr\fint_{B_{\rhob}(x_0)}\abs{\frac{u-\avgsueta[\rho_b]}{\rhob-\rhoa}}^2\ dx\rbr^\frac{1}{2}\ dt\\
				&\overset{\text{\cref{sobolev_bound}}}{\apprle}&\la^{p-1}\lbr\sup_{I_{\rhob}^\la(t_0)}\fint_{B_{\rhob}(x_0)}\abs{\frac{u-\avgsueta[\rho_b]}{\rhob-\rhoa}}^2\eta^p\zeta^2\ dx\rbr^\frac{1}{2}.
				\end{array}
			\end{equation*}
			
			Therefore applying Young's inequality, for any $\ve \in (0,1)$, we obtain 
			\begin{equation*}
			   \begin{array}{rcl}
			   II&\apprle& \ve \la^{p-2}\sup_{I_{\rhob}^\la(t_0)}\fint_{B_{\rhob}(x_0)}\abs{\frac{u-\avgsueta[\rho_b]}{\rhob-\rhoa}}^2\eta^p\zeta^2\ dx+\frac{1}{\ve}\la^{p}.
			   \end{array}			   
			\end{equation*}
		
    \item[Estimate for the third term:] This is an important term to estimate and we proceed as follows:
    \begin{equation*}
		   	\begin{array}{rcl}
		   	&&\fiint_{Q_{\rhob}^\la(z_0)}|f||u-\avgsueta[\rho_b]|\ dz\\
		   	&&\leq\fint_{I_{\rhob}^{\la}(t_0)}\lbr\fint_{B_{\rhob}(x_0)}(\rhob)^{\mfa'}|f|^{\mfa'}\ dx\rbr^\frac{1}{\mfa'}\lbr\fint_{B_{\rhob}(x_0)}\abs{\frac{u-\avgsueta[\rho_b]}{\rhob}}^{\mfa}\ dx\rbr^\frac{1}{\mfa}\ dt\\
		   	&&\overset{\text{\cref{sobolev_bound}}}{\apprle}\lbr\fint_{B_{\rhob}(x_0)}(\rhob)^{\mfa'}|f|^{\mfa'}\ dx\rbr^\frac{1}{\mfa'}\la
		   	\overset{\cref{le_la}}{\le}\la^{p}.
		   	\end{array}
		   	\end{equation*}
\end{description}
Thus combining the above estimates into \cref{caccio} and  recalling the restrictions on  $\rho \leq \rhoa \leq 2\rho$ and $3\rho < \rhob \leq 4\rho$, we get
\[
    \la^{p-2}\sup_{I_{\rhob}^\la(t_0)}\fint_{B_{\rhob}(x_0)}\abs{\frac{u-\avgsueta[\rho_b]}{\rhob}}^2\eta^p \zeta^2\ dx \apprle \ve \la^{p-2}\sup_{I_{\rhob}^\la(t_0)}\fint_{B_{\rhob}(x_0)}\abs{\frac{u-\avgsueta[\rho_b]}{\rhob}}^2\eta^p\zeta^2\ dx+\frac{1}{\ve}\la^{p}.
\]
Choosing $\ve$ small enough, the proof of the lemma follows.
\end{proof}
\begin{corollary}\label{cor_sup_2_bnd}
    With notation as in \cref{sup_2_bnd}, there holds
    \begin{equation*}
		 \sup_{I_{s}^\la(t_0)}\frac{1}{|B_r(x_0)|}\int_{B_{s}(x_0)}\abs{\frac{u-\avgs{u}{B_s(x_0)}(t)}{r}}^2 \ dx \apprle \la^2.
	\end{equation*}
\end{corollary}
\begin{proof}
    From triangle inequality, we have that $\int_{A} |f - \avgs{f}{A}|^q \ dx \leq 2^{q-1}\inf_{z \in \RR} \int_A |f - z|^q \ dx$ where $A$ is some measurable set, we obtain the following sequence of inequalities:
    \begin{equation*}
        \begin{array}{rcl}
            \sup_{I_{s}^\la(t_0)}\frac{1}{|B_r(x_0)|}\int_{B_{s}(x_0)}\abs{\frac{u-\avgs{u}{B_s(x_0)}(t)}{r}}^2 \ dx & \apprle & \sup_{I_{s}^\la(t_0)}\frac{1}{|B_r(x_0)|}\int_{B_{s}(x_0)}\abs{\frac{u-\avgs{u}{B_r(x_0)}(t)}{r}}^2 \ dx \\
            & \apprle & \sup_{I_{r}^\la(t_0)}\fint_{B_{r}(x_0)}\abs{\frac{u-\avgs{u}{B_r(x_0)}(t)}{r}}^2\eta^p \zeta^2 \ dx \\
            & \overset{\text{\cref{weight_lem}}}{\apprle}& \sup_{I_{r}^\la(t_0)}\fint_{B_{r}(x_0)}\abs{\frac{u-\avgs{u}{\eta}^r(t)}{r}}^2\eta^p \zeta^2 \ dx \\
            & \overset{\text{\cref{sup_2_bnd}}}{\apprle}& \la^2.
        \end{array}
    \end{equation*}

\end{proof}

\subsection{Reverse H\"older inequality in an intrinsic cylinder}
\begin{lemma}\label{sig=p}
	Assuming \cref{caccio} holds and with notation as in \cref{sup_2_bnd}, for any $\ve \in (0,1)$, the following holds:
	\begin{equation*}
		\fiint_{Q_{2\rho}^\la(z_0)}\abs{\frac{u-\avgs{u}{B_{2\rho}(x_0)}(t)}{2\rho}}^p\ dz \apprle_{(n,p,\La_0,\La_1,\de)}\ve\la^p+C(\ve)\lbr\fiint_{Q_{2\rho}^{\la}(z_0)}|\na u|^q\ dz\rbr^\frac{p}{q}.
	\end{equation*}
	Recall that $q$ is from \cref{reverse_exp}.
\end{lemma}
\begin{proof}
    We apply \cref{g_n} with $\sig=p$, $q$ as defined in \cref{reverse_exp}, $r=2$ and $\vartheta=\frac{q}{p}$. It is easy to see that the choice of exponents satisfies the hypothesis of \cref{g_n} thanks to the restriction from \cref{reverse_exp}. In particular, we have
    \begin{equation*}
	-\frac{n}{2}\le\vartheta(1-\frac{n}{q})-(1-\vartheta)\frac{n}{2}\quad \Longleftarrow\quad  \frac{np}{n+2}\le q.
	\end{equation*}
    Therefore, we have the following sequence of estimates:
	\begin{equation*}
	\begin{array}{rcl}
		\fiint_{Q_{2\rho}^\la(z_0)}\abs{\frac{u-\avgs{u}{B_{2\rho}(x_0)}(t)}{2\rho}}^p\ dz
		&\overset{\redlabel{6.17a}{a}}{\le} &  \lbr \fiint_{Q_{2\rho}^{\la}(z_0)}|\na u|^q\ dz\rbr \lbr\sup_{I_{2\rho}^\la(t_0)}\fint_{B_{2\rho}(x_0)}\abs{\frac{u-\avgs{u}{B_{2\rho}(x_0)}(t)}{2\rho}}^2\ dx\rbr^{\frac{p(1-\vartheta)}{2}}\\
		&\overset{\redlabel{6.17c}{b}}{\le} & \la^{p(1-\vartheta)}\fiint_{Q_{2\rho}^{\la}(z_0)}|\na u|^q\ dz,
	\end{array}
	\end{equation*}
    where to obtain \redlabel{6.17a}{a}, we made use of \cref{g_n} followed by Poincar\'e inequality and to obtain \redlabel{6.17c}{b}, we made use of \cref{cor_sup_2_bnd} applied with $s= 2\rho$ and $r = 3\rho$. Finally, an application of Young's inequality gives the desired result.
\end{proof}

\begin{lemma}\label{sig=2}
	Assuming \cref{caccio} holds and with notation as in \cref{sup_2_bnd}, for any $\ve \in (0,1)$, the following holds:
	\begin{equation*}
	\la^{p-2}\fiint_{Q_{2\rho}^\la(z_0)}\abs{\frac{u-\avgsutrho}{2\rho}}^2\ dz\apprle_{(n,p,\La_0,\La_1,\de)} \ve\la^p+C(\ve)\lbr\fiint_{Q_{2\rho}^\la(z_0)}|\na u|^q\ dz\rbr^\frac{p}{q}.
	\end{equation*}
	Recall that $q$ is from \cref{reverse_exp}.
\end{lemma}

\begin{proof}
	We apply \cref{g_n} with $\sig=2$, $q$ as defined in \cref{reverse_exp}, $r=2$ and any $\vartheta\in(0,1)$ is admissible since 
	\begin{equation*}
	-\frac{n}{2}\le\vartheta(1-\frac{n}{q})-(1-\vartheta)\frac{n}{2}\quad \Longleftarrow\quad  \frac{2n}{n+2}\le q.
	\end{equation*}
	Let us take $\vartheta\in(0,1)$ such that	$0<\vartheta <\frac{q}{2}$, then we get
	\begin{equation*}
	\begin{array}{rcl}
	&&\fiint_{Q_{2\rho}^\la(z_0)}\abs{\frac{u-\avgsutrho}{2\rho}}^2\ dz\\
	&&\apprle \fint_{I_{2\rho}^\la(t_0)}\lbr\fint_{B_{2\rho}(x_0)}\abs{\frac{u-\avgsutrho}{2\rho}}^q+|\na u|^q\ dx\rbr^\frac{2\vartheta}{q}\lbr\fint_{B_{2\rho}(x_0)}\abs{\frac{u-\avgsutrho}{2\rho}}^2\ dx\rbr^{1-\vartheta}\ dt\\
	&&\overset{\redlabel{6.23a}{a}}{\le}\lbr\fiint_{Q_{2\rho}^\la(z_0)}|\na u|^q\ dz\rbr^\frac{2\vartheta}{q}\lbr\sup_{I_{2\rho}^\la(t_0)}\fint_{B_{2\rho}(x_0)}\abs{\frac{u-\avgsutrho}{2\rho}}^2\ dx\rbr^{1-\vartheta}\\
	&&\overset{\redlabel{6.23b}{b}}{\apprle}\la^{2(1-\vartheta)}\lbr\fiint_{Q_{2\rho}^\la(z_0)}|\na u|^q\ dz\rbr^\frac{2\vartheta}{q}.
	\end{array}
	\end{equation*}
	Here, to obtain \redref{6.23a}{a}, we used Poincar$\acute{\text{e}}$ inequality, and to \redref{6.23b}{b}, we used \cref{cor_sup_2_bnd} applied with $s= 2\rho$ and $r= 3\rho$. Finally, an application of Young's inequality gives the desired estimate.
\end{proof}

\begin{lemma}\label{sig=mfa}
	Assuming \cref{caccio} holds and with notation as in \cref{sup_2_bnd}, for any $\ve \in (0,1)$, the following holds:
	\begin{equation*}
		\fiint_{Q_{2\rho}^\la(z_0)}|f||u-\avgsutrho|\ dz\apprle_{(n,p,\La_0,\La_1,\de)}\ve \la^p+C(\ve)\lbr\fiint_{Q_{2\rho}^\la(z_0)}|\na u|^{q} \ dz\rbr^\frac{p}{q}.
	\end{equation*}
	Recall that $q$ is from \cref{reverse_exp}.
\end{lemma}
\begin{proof} From H\"older's inequality, we have
\begin{equation}\label{5.23}
		\begin{array}{rcl}
	    &&\fiint_{Q_{2\rho}^\la(z_0)}|f||u-\avgsutrho|\ dz\\
	     &&\le\lbr\fint_{B_{2\rho}(x_0)}(2\rho)^{\mfa'}|f|^{\mfa'}\ dx\rbr^\frac{1}{\mfa'}\lbr \fint_{I_{2\rho}^{\la}(t_0)}\lbr\fint_{B_{2\rho}(x_0)} \abs{\frac{u-\avgsutrho}{2\rho}}^{\mfa} \ dx\rbr^\frac{1}{\mfa}\ dt \rbr\\\
       	&&\overset{\cref{le_la}}{\le}\la^{p-1}\lbr \fint_{I_{2\rho}^{\la}(t_0)}\lbr\fint_{B_{2\rho}(x_0)} \abs{\frac{u-\avgsutrho}{2\rho}}^{\mfa} \ dx\rbr^\frac{1}{\mfa}\ dt \rbr.
		\end{array}
		\end{equation}
		In order to estimate the integral term appearing on the right-hand side of \cref{5.23}, we consider two subcases and proceed as follows.
	\begin{description}
		\item [Case $p \leq n$:] In this case,  we can directly apply \cref{sob_poin} to \cref{5.23} to get
		\begin{equation*}
	    \lbr \fint_{I_{2\rho}^{\la}(t_0)}\lbr\fint_{B_{2\rho}(x_0)} \abs{\frac{u-\avgsutrho}{2\rho}}^{\mfa} \ dx\rbr^\frac{1}{\mfa}\ dt \rbr
	     \apprle\lbr\fiint_{Q_{2\rho}^\la(z_0)}|\na u|^{q} \ dz\rbr^\frac{1}{q}.
		\end{equation*}
		
		\item [Case $p> n$:] We will further restrict $\de$ such that $\de > \frac{n}{p}$ which implies $q >n$. Similarly, calculation from \cref{sobolev_bound} gives
		\begin{equation*}
			\begin{array}{rcl}
			\fiint_{Q_{2\rho}^\la(z_0)}|f||u-\avgsutrho|\ dz
			&\le& \fint_{I_{2\rho}^{\la}(t_0)}\lbr\fint_{B_{2\rho}(x_0)}(2\rho)^{\mfa'}|f|^{\mfa'}\ dx\rbr^\frac{1}{\mfa'}\lbr \sup_{B_{2\rho}(x_0)}\abs{\frac{u-\avgsutrho}{2\rho}}\rbr \ dt\\
			&\overset{\redlabel{5.27a}{a}}{\apprle}&\fint_{I_{2\rho}^{\la}(t_0)}\lbr\fint_{B_{2\rho}(x_0)}(2\rho)^{\mfa'}|f|^{\mfa'}\ dx\rbr^\frac{1}{\mfa'} \lbr\fint_{B_{2\rho}(x_0)}|\na u|^{q} \ dx\rbr^\frac{1}{q}\ dt \\
			&\overset{\redlabel{5.27b}{b}}{\apprle}&\la^{p-1}\lbr\fiint_{Q_{2\rho}^\la(z_0)}|\na u|^{q} \ dz\rbr^\frac{1}{q},
			\end{array}
		\end{equation*}
		where to obtain \redref{5.27a}{a}, we proceeded analogous to \cref{infity_estimate} noting that $q >n$ by the restriction on $\de$ and to obtain \redref{5.27b}{b}, we made use of the fact that $f$ is independent of time along with \cref{le_la}. 
	\end{description}
Applying Young's inequality in both cases completes the  proof.
\end{proof}
\begin{proposition}\label{intr_reverse}
	Under the assumption \cref{le_la} and \cref{ge_la} with $\la\ge1$,  there holds
	\begin{equation*}
		\fiint_{Q_{\rho}^\la(z_0)}(|\na u|+1)^p\ dz \apprle_{(n,p,\La_0,\La_1,\de)}\lbr\fiint_{Q_{2\rho}^\la(z_0)}(|\na u|+1)^q\ dz\rbr^\frac{p}{q}+\la^{\nonfe}\fint_{B_{2\rho}(x_0)}(2\rho)^{\mfa'}|f|^{\mfa'}\ dx.
	\end{equation*}
\end{proposition}
\begin{proof}
	Let us apply \cref{caccio} with $\rhoa=\rho$ and $\rhob=2\rho$, then along with \cref{weight_lem}, we obtain 
	\begin{equation*}
		\begin{array}{rcl}
		\fiint_{Q_{\rho}^\la (z_0)}|\na u|^p\ dz
		&\apprle_{(n,p,\La_0,\La_1)}& \fiint_{Q_{2\rho}^\la(z_0)}\abs{\frac{u-\avgsutrho}{2\rho}}^p\ dz+\la^{p-2}\fiint_{Q_{2\rho}^\la(z_0)}\abs{\frac{u-\avgsutrho}{2\rho}}^2\ dz\\
		&&+\fiint_{Q_{2\rho}^\la(z_0)}|f||u-\avgsutrho|\ dz.
		\end{array}
	\end{equation*}
	We estimate the right-hand side of the above expression using  \cref{sig=p}, \cref{sig=2} and \cref{sig=mfa} to get
	\begin{equation*}
		\fiint_{Q_{\rho}^\la(z_0)}|\na u|^p\ dz\apprle_{(n,p,\La_0,\La_1,\de)}\ve\la^p+C(\ve)\lbr\fiint_{Q_{2\rho}^\la(z_0)}|\na u|^{q} \ dz\rbr^\frac{p}{q},
	\end{equation*}
	from which we get
	\begin{equation*}
	\begin{array}{rcl}
	\fiint_{Q_{\rho}^\la(z_0)}(|\na u|+1)^p\ dz+\la^{\nonfe}\fint_{B_{\rho}(x_0)}\rho^{\mfa'}|f|^{\mfa'}\ dx
	&\apprle&\ve\la^p+C(\ve)\lbr\fiint_{Q_{2\rho}^\la(z_0)}(|\na u|+1)^{q} \ dz\rbr^\frac{p}{q}\\
	&&\qquad+\la^{\nonfe}\fint_{B_{2\rho}(x_0)}(2\rho)^{\mfa'}|f|^{\mfa'}\ dx.
	\end{array}
	\end{equation*}
	Now, taking $\ve=\ve(n,p,\La_0,\La_1,\de)\in(0,1)$ small enough and making use of \cref{ge_la}, we obtain the desired estimate.
\end{proof}

\subsection{Fixing some Constants}
\label{fixing_some_constants}

Let us fix some constants that will be needed in the covering argument to prove higher integrability. 
\begin{definition}\label{definition_constants}
    Let $r>0$ be fixed with $Q_{2r}(z_0) \subset \Om_T$ and let $r \leq r_1 < r_2 \leq 2r$ be any two radii, then we define the following quantities:
    \begin{description}
        \descitem{C1}{C1} For any $ \la >0$, we denote $E_{\la}:=\{z\in Q_{r_1}(z_0):|\na u|+1>\la\}$.
        \descitem{C2}{C2} We define the two constants 
                            \begin{gather*}
			d:=
			\begin{cases}
			\frac{p}{2}&\text{ if }p\ge2,\\
			\frac{2p}{p(n+2)-2n}&\text{ if }p<2,
			\end{cases}
			\txt{and}
			\alpha:=
			\begin{cases}
			\lbr d-\frac{d}{p}\nonfe\rbr^{-1}&\text{ if }p\ge2,\\
			\lbr1-\frac{d}{p}\mfa'\lbr\frac{p}{2}-\frac{p}{\mfa}\rbr\rbr^{-1}&\text{ if }p<2.
			\end{cases}
		\end{gather*}
        \descitem{C3}{C3} Fix $\la_0$ such that  $\la_0^\frac{p}{d}:=\fiint_{Q_{2r}(z_0)}\lbr|\na u|+1\rbr^p\ dz+\lbr\fint_{B_{2r}(x_0)}(2r)^{\mfa'}|f|^{\mfa'}\ dx\rbr^\alpha$.
        \descitem{C4}{C4} We fix a constant $\mathbb{B}$ satisfying $\mathbb{B}:=\lbr\frac{20r}{r_2-r_1}\rbr^{(n+2)\frac{d}{p}}+\lbr\frac{20r}{r_2-r_1}\rbr^{n\frac{d\alpha}{p}}\ge1$.
        \descitem{C5}{C5} For $(\mfx,\mft)=\mathfrak{z}\in E_\la$ and $\rho\in(0,r_2-r_1)$, we define
		\begin{equation*}
		G(Q_\rho^\la(\mathfrak{z})):=\fiint_{Q_\rho^\la(\mathfrak{z})}\lbr|\na u|+1\rbr^p\ dz+\la^{\nonfe}\fint_{B_\rho^\la(\mathfrak{x})}\rho^{\mfa'}|f|^{\mfa'}\ dx.
		\end{equation*}
        \descitem{C6}{C6} We define 
		$\nu:=\mfa'\lbr\frac{\min\{2,p\}}{2}-\frac{p}{\mfa}\rbr$ and $\tf(x):=\eta^{-\frac{\mfa'}{n-\mfa'}}\la_0^{\nu}(2r)^{\mfa'}|f(x)|^{\mfa'}$.
    \end{description}

\end{definition}

\begin{definition}\label{lv}
	Let $\eta>0$ and $\la>1$. For any $\rho\in(0,2r)$, we define the following upper-level sets:
	\begin{itemize}
		\item $\Phi_{\eta\la}^\rho:=\{z\in Q_{\rho}(z_0):(|\na u(z)|+1)^p>\eta\la^p\}$,
		\item $\Sig_{\eta\la}^{\rho}:=\{z\in Q_{\rho
		}(z_0): \eta^{-\frac{\mfa'}{n-\mfa'}}\la_0^{\nu}(2r)^{\mfa'}|f(x)|^{\mfa'}>\eta\la^p\}=\{z\in Q_{\rho
	}(z_0):|\tf|>\eta\la^p\}$,
	\end{itemize}
	 where $\tf$ and $\nu$ are as defined in \descref{C6}{C6}.
	
\end{definition}

Before we end this subsection, we make an important remark:
\begin{remark}\label{singular_cylinder}
    In the rest of the proof of the higher integrability, we will consider the following cylinders:
    \begin{description}
        \item[Case $p \geq 2$:] The cylinder will be of the form $Q_r^{\la} := Q_{r,\la^{2-p}r^2}$.
        \item[Case $p \leq 2$:] The cylinder will be of the form $Q_r^{\la} := Q_{\la^{\frac{p-2}{2}}r,r^2}$. 
    \end{description}
    By abuse of notation, we shall use the same notation $Q_r^{\la}$ to denote these two different cylinders and their notation will be clear from the context.
\end{remark}

\subsection{Covering argument in the case \texorpdfstring{$p \geq 2$}.}
\label{notation_covering_argument}

\begin{lemma}
 \label{stopping}
    With the constants fixed as in \cref{fixing_some_constants}, let $\la > 2^{\max\left\{\tfrac{d}{p}, \tfrac{\alpha d}{p}\right\}} \mathbb{B} \la_0$ and $\mfz \in E_{\la}$, then there exists $\rhoz\in\lbr0,\tfrac{r_2-r_1}{10}\rbr$ such that the following holds
    \begin{equation*}
	G(Q_{\rhoz}^\la(\mathfrak{z}))=\la^p\txt{and} G(Q_{\rho}^\la(\mathfrak{z}))\le\la^p\ \text{ for all }\ \ \rho\in(\rhoz,r_2-r_1),
	\end{equation*}
	where $G$ is defined in \descref{C5}{C5}.
\end{lemma}

\begin{proof}
From the Lebesgue differentiation theorem, we see that 
\[
    \lim_{\rho \rightarrow 0} G(Q_\rho^\la(\mathfrak{z})) \geq (|\nabla u(\mfz)|+ 1)^p \overset{\descref{C1}{C1}}{>} \la^p.
\]

	From the observation $\frac{p}{d \alpha} > 0$, we see that for any $\frac{r_2-r_1}{10}\leq \rho<r_2-r_1$, it follows that $Q_\rho^\la(\mathfrak{z})\subset Q_{2r}(z_0)$ and 
		\begin{equation*}
			\begin{array}{rcl}
			G(Q_\rho^\la(\mathfrak{z}))
			&\le&\frac{|Q_{2r}(z_0)|}{|Q_{\rho}^\la(z_0)|}\fiint_{Q_{2r}(z_0)}\lbr|\na u|+1\rbr^p\ dz+\frac{|B_{2r}(x_0)|}{|B_{\rho}(x_0)|}\la^{\nonfe}\fint_{B_{2r}(x_0)}(2r)^{\mfa'}|f|^{\mfa'}\ dx\\
			&\overset{\descref{C3}{C3}}{\le}&\lbr\frac{20r}{r_2-r_1}\rbr^{n+2}\la^{p-2}\la_0^{\frac{p}{d}}+\lbr\frac{20r}{r_2-r_1}\rbr^{n}\la^{\nonfe}\la_0^{\frac{p}{d}\frac{1}{\alpha}}\\
			&\overset{\descref{C2}{C2},\descref{C4}{C4}}{\le}&\frac{1}{2}\la^p+\frac{1}{2}\la^p=\la^p.
			\end{array}
		\end{equation*}		
		
		Let us define $\rhoz\in\lbr0,\frac{r_2-r_1}{10}\rbr$ to be the largest value such that $G(Q_{\rhoz}^\la(\mathfrak{z}))=\la^p$, this proves the lemma.
\end{proof}

\begin{lemma}\label{lemma5.17}
There exists $\eta \in (0,1)$ small such that for any $\la>2^{\max\left\{\tfrac{d}{p}, \tfrac{\alpha d}{p}\right\}}\mathbb{B}\la_0$ and $\mfz\in E_\la$, the following holds:
     \begin{equation}\label{step1_est}
     	\iint_{Q_{10\rhoz}^\la(\mathfrak{z})}(|\na u|+1)^p\ dz \apprle_{(n,p,\La_0,\La_1,\de)} \iint_{Q_{2\rhoz}^\la(\mfz)\cap\Phi_{\eta\la}^{r_2}}\la^{p-q}(|\na u|+1)^q\ dz+\iint_{Q_{2\rhoz}^\la(\mfz)\cap \Sig_{\eta\la}^{r_2}}|\tf(x)|\ dz.
     \end{equation}    
\end{lemma}

\begin{proof}
    From \cref{stopping}, we see that  assumptions \cref{le_la} and \cref{ge_la} are satisfied since $10\rhoz < r_2-r_1$. Thus \cref{intr_reverse} gives
	\begin{equation}\label{pre_est1}
	\fiint_{Q_{\rhoz}^\la(\mathfrak{z})}(|\na u|+1)^p\ dz \apprle\underbrace{\lbr\fiint_{Q_{2\rhoz}^\la(\mathfrak{z})}(|\na u|+1)^q\ dz\rbr^\frac{p}{q}}_{=: I}+\underbrace{\la^{\nonfe}\fint_{B_{2\rhoz}^\la(\mathfrak{z})}(2\rhoz)^{\mfa'}|f|^{\mfa'}\ dx}_{=:II}.
	\end{equation}
In order to estimate \cref{pre_est1}, we will first choose some $\eta \in (0,1)$ to be eventually fixed depending on the data and estimate each of the terms as follows:
	\begin{description}
		\item [Estimate of $I$:] Let $\eta\in(0,1)$ will be chosen later. Then since $Q_{\rhoz}^\la(\mfz)\subset Q_{r_2}(z_0)$ holds, we see that
		\begin{equation}\label{est_I_5.37}
			\begin{array}{rcl}
			I
			&\overset{\text{\cref{lv}}}{\le}&\eta\la^p+\lbr\frac{1}{|Q_{2\rhoz}^\la(\mfz)|}\iint_{Q_{2\rhoz}^\la(\mfz)\cap \Phi_{\eta\la}^{r_2}}(|\na u|+1)^q\ dz\rbr^\frac{p}{q}\\
			&\le&\eta\la^p+\lbr\frac{1}{|Q_{2\rhoz}^\la(\mfz)|}\iint_{Q_{2\rhoz}^\la(\mfz)\cap\Phi_{\eta\la}^{r_2}}(|\na u|+1)^q\ dz\rbr\lbr\fiint_{Q_{2\rhoz}^\la(\mfz)}(|\na u|+1)^q\ dz\rbr^{\frac{p}{q}-1}\\
			&\overset{\cref{le_la}}{\le}&\eta\la^p+\frac{1}{|Q_{2\rhoz}^\la(\mfz)|}\iint_{Q_{2\rhoz}^\la(\mfz)\cap\Phi_{\eta\la}^{r_2}}\la^{p-q}(|\na u|+1)^q\ dz.
			\end{array}
		\end{equation}

		\item [Estimate of $II$:] To estimate this term, we shall split the proof into two cases as follows (note that $\eta$ is yet to be determined):
\begin{description}[leftmargin=0pt]
\item{{\bf Case a}:} In this case, let us assume that the following holds:
\begin{gather}\label{alt}
			\eta\la^p\le\la^{\nonfe}\fint_{B_{2\rhoz}^\la(\mathfrak{z})}(2\rhoz)^{\mfa'}|f|^{\mfa'}\ dx.
		\end{gather}
		Recall that $f = f(x)$, thus we get
		\begin{equation}\label{5.35}
				\eta\la^p
				\le\la^{\nonfe}\fint_{B_{2\rhoz}(\mathfrak{x})}(2\rhoz)^{\mfa'}|f|^{\mfa'}\ dx
				\overset{\text{\descref{C3}{C3}}}{\le}\la^{\nonfe}\lbr\frac{r}{\rhoz}\rbr^{n-\mfa'}\la_0^{\frac{p}{d\alpha}}.
			\end{equation}
			It is easy to see that $n > \mfa'$, since, in the case $p>n$, we trivially have $n > \mfa'=1$ (see \cref{mfa_high}) and in the case $p\leq n$, we have $n > \mfa' \Leftrightarrow \mfa > n' \Leftrightarrow \de > \frac{1}{p}$ which holds from \cref{mfa_high}.
			
			From \cref{5.35} and the fact that $\frac{p}{d\al} =p - \mfa'\lbr 1-\frac{p}{\mfa}\rbr =\mfa'(p-1)> 0$ along with \descref{C2}{C2}, we have 
			\begin{equation}\label{alt_est1}
				\frac{\rhoz}{r}\le \eta^{-\frac{1}{n-\mfa'}}\lbr \frac{\la_0}{\la}\rbr^{\frac{1}{n-\mfa'}\lbr p-\nonfe\rbr}.
			\end{equation}
		Therefore \cref{alt_est1} implies that 
		\begin{equation}\label{alt_est2}
			\begin{array}{rcl}
				\la^{\nonfe}\fint_{B_{2\rhoz}^\la(\mfz)}(2\rhoz)^{\mfa'}|f|^{\mfa'}\ dx
				&=&\lbr\frac{\rhoz}{r}\rbr^{\mfa'}\la^{\nonfe}\fint_{B_{2\rhoz}^\la(\mfz)}(2r)^{\mfa'}|f|^{\mfa'}\ dx\\
				&\le& \lbr \frac{\la_0}{\la} \rbr^{\frac{\mfa'}{n-\mfa'}\lbr p-\nonfe\rbr-\nonfe}\frac{\la_0^{\nonfe}}{\eta^{\frac{\mfa'}{n-\mfa'}}}\fint_{B_{2\rhoz}^\la(\mfz)}(2r)^{\mfa'}|f|^{\mfa'}\ dx\\
				& \overset{\redlabel{5.37a}{a}}{\leq} & \eta^{-\frac{\mfa'}{n-\mfa'}}\la_0^{\nonfe}\fint_{B_{2\rhoz}^\la(\mfz)}(2r)^{\mfa'}|f|^{\mfa'}\ dx\overset{\text{\descref{C6}{C6}}}{=}\fiint_{Q_{2\rhoz}^\la(\mfz)}|\tf|\ dz\\
				&\overset{\redlabel{5.37b}{b}}{\leq}&\eta\la^p+\frac{1}{|B_{2\rhoz}^\la(\mfz)|}\int_{B_{2\rhoz}^\la(\mfz)\cap \Sig_{\eta\la}^{r_2}}|\tf|\ dx,
			\end{array}
		\end{equation}
	where to obtain \redref{5.37a}{a}, we noted that $\lbr\frac{\mfa'}{n-\mfa'}\rbr\lbr p-\nonfe\rbr-\nonfe >0$ which follows from \cref{mfa_high} and the restriction $\la > \la_0$; to obtain \redref{5.37b}{b}, we made use of \cref{lv} along with \descref{C6}{C6}.
\item{{\bf Case b}:} This is the trivial case where \cref{alt} does not hold, i.e., the following holds:
\begin{gather}\label{alt2}
			 \la^{\nonfe}\fint_{B_{2\rhoz}^\la(\mathfrak{z})}(2\rhoz)^{\mfa'}|f|^{\mfa'}\ dx \leq \eta\la^p.
		\end{gather}
\end{description}		
	\end{description}
Combining \cref{est_I_5.37} and \cref{alt_est2} or \cref{alt2} into \cref{pre_est1}, we get
		\begin{equation*}
			\fiint_{Q_{\rhoz}^\la(\mathfrak{z})}(|\na u|+1)^p\ dz \apprle \eta\la^p+ \frac{1}{|Q_{2\rhoz}^\la(\mfz)|}\iint_{Q_{2\rhoz}^\la(\mfz)\cap\Phi_{\eta\la}^{r_2}}\la^{p-q}(|\na u|+1)^q\ dz+\frac{1}{|B_{2\rhoz}^\la(\mfz)|}\int_{B_{2\rhoz}^\la(\mfz)\cap \Sig_{\eta\la}^{r_2}}|\tf|\ dx.
		\end{equation*}
	Recalling \descref{C5}{C5}, \cref{stopping} and \cref{pre_est1}, we get
	\begin{equation}\label{step1_est_1}
		G(Q_{\rhoz}^\la(\mfz))\apprle \eta G(Q_{\rhoz}^\la(\mfz))+ \frac{1}{|Q_{2\rhoz}^\la(\mfz)|}\iint_{Q_{2\rhoz}^\la(\mfz)\cap\Phi_{\eta\la}^{r_2}}\la^{p-q}(|\na u|+1)^q\ dz+\frac{1}{|B_{2\rhoz}^\la(\mfz)|}\int_{B_{2\rhoz}^\la(\mfz)\cap \Sig_{\eta\la}^{r_2}}|\tf|\ dx.
	\end{equation}
We now choose $\eta$ small enough to absorb $G(Q_{\rhoz}^\la(\mfz))$ and recalling the bound $G(Q_{10\rhoz}^\la(\mfz))<\la^p=G(Q_{\rhoz}^\la(\mfz))$ which holds due to \cref{stopping}, we get
\begin{equation*}
	\iint_{Q_{10\rhoz}^\la(\mathfrak{z})}(|\na u|+1)^p\ dz \apprle_{(n,p,\La_0,\La_1,\de)} \iint_{Q_{2\rhoz}^\la(\mfz)\cap\Phi_{\eta\la}^{r_2}}\la^{p-q}(|\na u|+1)^q\ dz+ \iint_{Q_{2\rhoz}^\la(\mfz)\cap \Sig_{\eta\la}^{r_2}}|\tf(x) |\ dx.
\end{equation*}
		Note that when we multiply \cref{step1_est_1} with $|Q_{2\rhoz}^\la(\mfz)|$, the last term in \cref{step1_est_1} becomes  $$ |I_{2\rhoz}^{\la}|\int_{_{2\rhoz}^\la(\mfz)\cap \Sig_{\eta\la}^{r_2}}|\tf|\ dx = \iint_{Q_{2\rhoz}^\la(\mfz)\cap \Sig_{\eta\la}^{r_2}}|\tf(x)|\ dz.$$ 
	\end{proof}

	\subsection{Covering argument in the case \texorpdfstring{$p<2$}.}
	Recall the notation from \cref{singular_cylinder} which will be used in this subsection dealing with the singular case. In particular, we take $\rho = \la^{\frac{p-2}{2}}\rho$.
	\begin{lemma}
		\label{stopping_sing}
		With the constants fixed as in \cref{fixing_some_constants}, let $\la > 2^{\max\left\{\tfrac{d}{p}, \tfrac{\alpha d}{p}\right\}} \mathbb{B} \la_0$ and $\mfz \in E_{\la}$, then there exists $\rhoz\in\lbr0,\tfrac{r_2-r_1}{10}\rbr$ such that following holds:
		\begin{equation*}
		G(Q_{\rhoz}^\la(\mathfrak{z}))=\la^p\txt{and} G(Q_{\rho}^\la(\mathfrak{z}))\le\la^p\ \text{ for all }\ \ \rho\in(\rhoz,r_2-r_1),
		\end{equation*}
		where $G$ is defined in \descref{C5}{C5}.
	\end{lemma}
	
	\begin{proof}
		From the Lebesgue differentiation theorem, we see that 
		\[
		\lim_{\rho \rightarrow 0} G(Q_\rho^\la(\mathfrak{z})) \geq (|\nabla u(\mfz)|+ 1)^p \overset{\descref{C1}{C1}}{>} \la^p.
		\]
		
		From observation $\frac{p}{d \alpha} > 0$, we see that for any $\frac{r_2-r_1}{10}\leq \rho<r_2-r_1$, it follows that $Q_\rho^\la(\mathfrak{z})\subset Q_{2r}(z_0)$. Recalling that the cylinders are of the form $Q_{\la^{\frac{p-2}{2}}R,R^2}$, we get 
		\begin{equation*}
		\begin{array}{rcl}
		G(Q_\rho^\la(\mathfrak{z}))
		&\le&\frac{|Q_{2r}(z_0)|}{|Q_{\rho}^\la(z_0)|}\fiint_{Q_{2r}(z_0)}\lbr|\na u|+1\rbr^p\ dz+\frac{|B_{2r}(x_0)|}{|B_{\rho}^{\la}(x_0)|}\la^{\nonfe}\fint_{B_{2r}(x_0)}(2r\la^{\frac{p-2}{2}})^{\mfa'}|f|^{\mfa'}\ dx\\
		&\overset{\descref{C3}{C3}}{\le}&\lbr\frac{20r}{r_2-r_1}\rbr^{n+2}\la^{\frac{(2-p)n}{2}}\la_0^{\frac{p(n+2)-2n}{2}}+\lbr\frac{20r}{r_2-r_1}\rbr^{n}\la^{\nonfe+\mfa' \lbr \frac{p-2}{2}\rbr}\la^{\frac{(2-p)n}{2}}\la_0^{\frac{p}{d}\frac{1}{\alpha}}\\
		&\overset{\descref{C2}{C2},\descref{C4}{C4}}{\le}&\frac{1}{2}\la^p+\frac{1}{2}\la^p=\la^p.
		\end{array}
		\end{equation*}		
		
		Let us define $\rhoz\in\lbr0,\frac{r_2-r_1}{10}\rbr$ to be the largest value such that $G(Q_{\rhoz}^\la(\mathfrak{z}))=\la^p$. This proves the lemma.
	\end{proof}

\begin{lemma}\label{lemma5.19}
	There exists $\eta \in (0,1)$ small such that for any $\la>2^{\max\left\{\tfrac{d}{p}, \tfrac{\alpha d}{p}\right\}}\mathbb{B}\la_0$ and $\mfz\in E_\la$, the following holds:
	\begin{equation*}
	\iint_{Q_{10\rhoz}^\la(\mathfrak{z})}(|\na u|+1)^p\ dz \apprle_{(n,p,\La_0,\La_1,\de)} \iint_{Q_{2\rhoz}^\la(\mfz)\cap\Phi_{\eta\la}^{r_2}}\la^{p-q}(|\na u|+1)^q\ dz+\iint_{Q_{2\rhoz}^\la(\mfz)\cap \Sig_{\eta\la}^{r_2}}|\tf(x)|\ dz.
	\end{equation*}    
\end{lemma}

\begin{proof}
	From \cref{stopping}, we see that assumptions \cref{le_la} and \cref{ge_la} are satisfied since $10\rhoz < r_2-r_1$. Thus \cref{intr_reverse} gives
	\begin{equation}\label{pre_est1_sing}
	\fiint_{Q_{\rhoz}^\la(\mathfrak{z})}(|\na u|+1)^p\ dz \apprle\underbrace{\lbr\fiint_{Q_{2\rhoz}^\la(\mathfrak{z})}(|\na u|+1)^q\ dz\rbr^\frac{p}{q}}_{=: I}+\underbrace{\la^{\nonfe + \mfa'\lbr \frac{p-2}{2}\rbr}\fint_{B_{2\rhoz}^\la(\mathfrak{z})}(2\rhoz)^{\mfa'}|f|^{\mfa'}\ dx}_{=:II}.
	\end{equation}
	In order to estimate \cref{pre_est1_sing}, we will first choose some $\eta \in (0,1)$ to be eventually fixed depending on the data and estimate each of the terms as follows:
	\begin{description}
		\item [Estimate of $I$:] Let $\eta\in(0,1)$ will be chosen later. Then since $Q_{\rhoz}^\la(\mfz)\subset Q_{r_2}(z_0)$ holds, we see that
		\begin{equation}\label{est_I_5.37_sing}
		\begin{array}{rcl}
		I
		&\overset{\text{\cref{lv}}}{\le}&\eta\la^p+\lbr\frac{1}{|Q_{2\rhoz}^\la(\mfz)|}\iint_{Q_{2\rhoz}^\la(\mfz)\cap \Phi_{\eta\la}^{r_2}}(|\na u|+1)^q\ dz\rbr^\frac{p}{q}\\
		&\le&\eta\la^p+\lbr\frac{1}{|Q_{2\rhoz}^\la(\mfz)|}\iint_{Q_{2\rhoz}^\la(\mfz)\cap\Phi_{\eta\la}^{r_2}}(|\na u|+1)^q\ dz\rbr\lbr\fiint_{Q_{2\rhoz}^\la(\mfz)}(|\na u|+1)^q\ dz\rbr^{\frac{p}{q}-1}\\
		&\overset{\cref{le_la}}{\le}&\eta\la^p+\frac{1}{|Q_{2\rhoz}^\la(\mfz)|}\iint_{Q_{2\rhoz}^\la(\mfz)\cap\Phi_{\eta\la}^{r_2}}\la^{p-q}(|\na u|+1)^q\ dz.
		\end{array}
		\end{equation}

		\item [Estimate of $II$:] To estimate this term, we shall split the proof into two cases as follows (note that $\eta$ is yet to be determined):
		\begin{description}[leftmargin=0pt]
			\item{{\bf Case a}:} In this case, let us assume that the following holds:
			\begin{equation*}
			\eta\la^p\le\la^{\nonfe+\mfa'\lbr \frac{p-2}{2}\rbr}\fint_{B_{2\rhoz}^\la(\mathfrak{z})}(2\rhoz)^{\mfa'}|f|^{\mfa'}\ dx.
			\end{equation*}
			Recall that $f = f(x)$, thus we get
			\begin{equation}\label{5.35_sing}
			\eta\la^p
			\le\la^{\nonfe+\mfa'\lbr \frac{p-2}{2}\rbr}\fint_{B_{2\rhoz}^{\la}(\mathfrak{x})}(2\rhoz)^{\mfa'}|f|^{\mfa'}\ dx
			\overset{\text{\descref{C3}{C3}}}{\le}\la^{\nonfe+\mfa'\lbr \frac{p-2}{2}\rbr + \frac{n(2-p)}{2}}\lbr\frac{r}{\rhoz}\rbr^{n-\mfa'}\la_0^{\frac{p}{d\alpha}}.
			\end{equation}
			It is easy to see that $n > \mfa'$, since in the case $p>n$, we trivially have $n > \mfa'=1$ (see \cref{mfa_high}) and in the case $p\leq n$, we have $n > \mfa' \Leftrightarrow \mfa > n' \Leftrightarrow \de > \frac{1}{p}$ which holds from \cref{mfa_high}.
			
			We see that $\frac{p}{d\alpha}>0$ if and only if $\al>0$ which is equivalent to $d < \frac{2(\mfa-1)}{\mfa-2}$. Since we are in the case $p \leq 2$ and $n \geq 2$, we see that $\mfa = \frac{n\de p}{n-\de p}$, thus $\al >0$ is equivalent to 
			\begin{equation*}
				\begin{array}{rcl}
					\frac{2p}{p(n+2)-2n} \leq 2 \lbr \frac{n\de p - n +\de p}{\de p(n+2)-2n}\rbr \quad \Longleftrightarrow \quad 0 &\leq& n(\de p  -1) + p \lbr \de - \frac{\de p (n+2)-2n}{p(n+2)-2n}\rbr\\
					& =& n(\de p -1) + \frac{2np(1-\de)}{ p (n+2)-2n}.
				\end{array}
			\end{equation*}
			Since we assumed $\de > \frac{1}{p}$ from \cref{mfa_high}, the above quantity is positive. Moreover, we see that 
			\[
			p-\nonfe-\mfa'\lbr \frac{p-2}{2}\rbr-\frac{n(2-p)}{2} = \frac{p}{d \al}, 
			\]
			which is used to rewrite \cref{5.35_sing} as 
			\begin{equation}\label{alt_est3}
			\frac{\tilde{\rhoz}}{r}\leq \eta^{-\frac{1}{n-\mfa'}}\lbr\frac{\la_0}{\la}\rbr^{\frac{p}{d\alpha(n-\mfa')}}.
			\end{equation}
			Therefore \cref{alt_est3} implies that 
			\begin{equation}\label{alt_est2_sing}
			\begin{array}{rcl}
			&&\la^{\nonfe + \mfa'\lbr \frac{p-2}{2}\rbr}\fint_{B_{2\rhoz}^\la(\mfz)}(2\rhoz)^{\mfa'}|f|^{\mfa'}\ dx\\
			&&=\lbr\frac{\rhoz}{r}\rbr^{\mfa'}\la^{\nonfe + \mfa'\lbr \frac{p-2}{2}\rbr}\fint_{B_{2\rhoz}^\la(\mfz)}(2r)^{\mfa'}|f|^{\mfa'}\ dx\\
			&&\le \lbr \frac{\la_0}{\la} \rbr^{\frac{\mfa'}{n-\mfa'} \frac{p}{d\al}-\nonfe - \mfa'\lbr \frac{p-2}{2}\rbr}\frac{\la_0^{\nonfe + \mfa'\lbr \frac{p-2}{2}\rbr}}{\eta^{\frac{\mfa'}{n-\mfa'}}}\fint_{B_{2\rhoz}^\la(\mfz)}(2r)^{\mfa'}|f|^{\mfa'}\ dx\\
			&& \overset{\redlabel{5.37sa}{a}}{\leq}  \eta^{-\frac{\mfa'}{n-\mfa'}}\la_0^{\nonfe + \mfa'\lbr \frac{p-2}{2}\rbr}\fint_{B_{2\rhoz}^\la(\mfz)}(2r)^{\mfa'}|f|^{\mfa'}\ dx\overset{\text{\descref{C6}{C6}}}{=}\fiint_{Q_{2\rhoz}^\la(\mfz)}|\tf|\ dz\\
			&&\overset{\redlabel{5.37sb}{b}}{\leq}\eta\la^p+\frac{1}{|B_{2\rhoz}^\la(\mfz)|}\int_{B_{2\rhoz}^\la(\mfz)\cap \Sig_{\eta\la}^{r_2}}|\tf|\ dx,
			\end{array}
			\end{equation}
			where to obtain \redref{5.37sa}{a}, we noted that $\la \geq \la_0$ along with the calculations (recall $\mfa = \tfrac{n\de p}{n-\de p}$)		
			\[
			\begin{array}{rcl}
			\frac{\mfa'}{n-\mfa'} \frac{p}{d\al}-\nonfe - \mfa'\lbr \frac{p-2}{2}\rbr >0  & \Longleftrightarrow & \frac{1}{n-\mfa'} \lbr \frac{1}{d} - n \lbr \frac12 - \frac{1}{\mfa}\rbr\rbr >0 \\
			& \Longleftrightarrow & \frac{n(1-\de)}{\de p} > 0,
			\end{array}
			\]
			and  to obtain \redref{5.37sb}{b}, we made use of \cref{lv} along with \descref{C6}{C6}.
			\item{{\bf Case b}:} This is the trivial case where \cref{alt} does not hold, i.e., the following holds:
			\begin{gather}\label{alt2_sing}
			\la^{\nonfe + \mfa'\lbr \frac{p-2}{2}\rbr}\fint_{B_{2\rhoz}^\la(\mathfrak{z})}(2\rhoz)^{\mfa'}|f|^{\mfa'}\ dx \leq \eta\la^p.
			\end{gather}
		\end{description}		
	\end{description}
	Combining \cref{est_I_5.37_sing} and \cref{alt_est2_sing} or \cref{alt2_sing} into \cref{pre_est1_sing}, we get
	\begin{equation*}
	\fiint_{Q_{\rhoz}^\la(\mathfrak{z})}(|\na u|+1)^p\ dz \apprle \eta\la^p+ \frac{1}{|Q_{2\rhoz}^\la(\mfz)|}\iint_{Q_{2\rhoz}^\la(\mfz)\cap\Phi_{\eta\la}^{r_2}}\la^{p-q}(|\na u|+1)^q\ dz+\frac{1}{|B_{2\rhoz}^\la(\mfz)|}\int_{B_{2\rhoz}^\la(\mfz)\cap \Sig_{\eta\la}^{r_2}}|\tf|\ dx.
	\end{equation*}
	Recalling \descref{C5}{C5}, \cref{stopping_sing} and \cref{pre_est1_sing}, we get
	\begin{equation}\label{step1_est_1_sing}
	G(Q_{\rhoz}^\la(\mfz))\apprle \eta G(Q_{\rhoz}^\la(\mfz))+ \frac{1}{|Q_{2\rhoz}^\la(\mfz)|}\iint_{Q_{2\rhoz}^\la(\mfz)\cap\Phi_{\eta\la}^{r_2}}\la^{p-q}(|\na u|+1)^q\ dz+\frac{1}{|B_{2\rhoz}^\la(\mfz)|}\int_{B_{2\rhoz}^\la(\mfz)\cap \Sig_{\eta\la}^{r_2}}|\tf|\ dx.
	\end{equation}
	We now choose $\eta$ small enough to absorb $G(Q_{\rhoz}^\la(\mfz))$ and recalling the bound $G(Q_{10\rhoz}^\la(\mfz))<\la^p=G(Q_{\rhoz}^\la(\mfz))$ which holds due to \cref{stopping_sing}, we get
	\begin{equation*}
	\iint_{Q_{10\rhoz}^\la(\mathfrak{z})}(|\na u|+1)^p\ dz \apprle_{(n,p,\La_0,\La_1,\de)} \iint_{Q_{2\rhoz}^\la(\mfz)\cap\Phi_{\eta\la}^{r_2}}\la^{p-q}(|\na u|+1)^q\ dz+ \iint_{Q_{2\rhoz}^\la(\mfz)\cap \Sig_{\eta\la}^{r_2}}|\tf(x) |\ dx.
	\end{equation*}
	Note that when we multiply \cref{step1_est_1_sing} with $|Q_{2\rhoz}^\la(\mfz)|$, the last term in \cref{step1_est_1_sing} becomes  $$ |I_{2\rhoz}^{\la}|\int_{_{2\rhoz}^\la(\mfz)\cap \Sig_{\eta\la}^{r_2}}|\tf|\ dx = \iint_{Q_{2\rhoz}^\la(\mfz)\cap \Sig_{\eta\la}^{r_2}}|\tf(x)|\ dz.$$ 
\end{proof}
	\subsection{Vitali Covering lemma}

	\begin{lemma}
		With $\eta$ fixed according to \cref{lemma5.17} and \cref{lemma5.19}, let $\la_1 = 2^{\max\left\{\tfrac{d}{p}, \tfrac{\alpha d}{p}\right\}} \mathbb{B} \la_0$, then for any $\la > \la_1$, there holds
		\begin{equation}\label{pre_lv_est}
			\iint_{\Phi_{\la}^{r_1}}(|\na u|+1)^p\ dz
			\apprle_{(n,p,\La_0,\La_1,\de)}\iint_{\Phi_{\la}^{r_2}}\la^{p-q}(|\na u|+1)^q\ dz+\iint_{\Sig_{\la}^{r_2}}|\tf|\ dz.
		\end{equation}
	\end{lemma}
\begin{proof}
	Applying Vitali covering Lemma to $\{Q_{10\rhoz}^\la(\mfz)\}_{\mfz\in E_\la}$ where the cylinders are as obtained in \cref{stopping} and \cref{stopping_sing}, then we obtain disjoint countable subfamily  $\{Q_{2\rho_{\mfz_i}}^\la(\mfz_i)\}_{\mfz_i\in E_\la}$ such that 
	\begin{equation*}
		E_\la\subset\bigcup_{\mfz\in E_\la}Q_{2\rhoz}^\la(\mfz)
		\subset\bigcup_{1\le i<\infty}Q_{10\rho_{\mfz_i}}^\la(\mfz_i)\subset Q_{r_2}(z_0),
	\end{equation*}
where the last inclusion holds since $10 \rho_{\mfz_i} \leq r_2-r_1$. Thus we get
\begin{equation*}
	\begin{array}{rcl}
		\iint_{E_\la}(|\na u|+1)^p\ dz
		&\le&\sum_{1\le i<\infty}\iint_{Q_{10\rho_{\mfz_i}^\la(\mfz_i)}}(|\na u|+1)^p\ dz\\
		&\overset{\cref{step1_est}}{\apprle}&\sum_{1\le i<\infty}\lbr\iint_{Q_{2\rho_{\mfz_i}^\la(\mfz_i)\cap \Phi^{r_2}_{\eta\la}}}\la^{p-q}(|\na u|+1)^q\ dz+\iint_{Q_{2\rho_{\mfz_i}^\la(\mfz_i)\cap\Sig_{\eta\la}^{r_2}}}|\tf|\ dz\rbr\\
		&\le&\iint_{\Phi_{\eta\la}^{r_2}}\la^{p-q}(|\na u|+1)^q\ dz+\iint_{\Sig_{\eta\la}^{r_2}}|\tf|\ dz.
	\end{array}
\end{equation*}
Also, since there holds
\begin{equation*}
	\iint_{\Phi_{\eta\la}^{r_1}\setminus E_\la}(|\na u|+1)^p\ dz\le\iint_{\Phi_{\eta\la}^{r_1}\setminus E_\la}\la^{p-q}(|\na u|+1)^q\ dz,
\end{equation*}
from which it follows that
\begin{equation*}
	\iint_{\Phi_{\eta\la}^{r_1}}(|\na u|+1)^p\ dz
	=\iint_{\Phi_{\eta\la}^{r_1}\setminus E_\la}(|\na u|+1)^p\ dz+\iint_{E_\la}(|\na u|+1)^p\ dz
	\apprle\iint_{\Phi_{\eta\la}^{r_2}}\la^{p-q}(|\na u|+1)^q\ dz+\iint_{\Sig_{\eta\la}^{r_2}}|\tf|\ dz.
\end{equation*}
\end{proof}
	
	\subsection{Proof of \texorpdfstring{\cref{thm_1}}.}\label{sub2.2}
\label{subsection5.8}
For $k>\la_1$, let us define
\begin{gather*}
	|\na u|_k:=\min\{|\na u|+1,k\}\txt{and}\Phi^{\rho}_{\la,k}:=\{z\in Q_{\rho}(z_0): |\na u|^p_k>\la^p\}.
\end{gather*}
From this, we see that for  $\la>k$, then $\Phi^{r_1}_{\la,k}=\varnothing$ and if $\la\le k$, then $\Phi^{r_1}_{\la,k}=\Phi^{r_1}_{\la}$ and $\Phi^{r_2}_{\la,k}=\Phi^{r_2}_{\la}$. Thus from \cref{pre_lv_est}, we deduce 
\begin{equation}\label{lv_est}
	\iint_{\Phi_{\la,k}^{r_1}}|\na u|^p\ dz
	\apprle\iint_{\Phi_{\la,k}^{r_2}}\la^{p-q}|\na u|^q\ dz+\iint_{\Sig_{\la}^{r_2}}|\tf|\ dz.
\end{equation}
Let $a>0$. Multiplying \cref{lv_est} with  $\frac{\la^{a-1}}{e+\la^a}$ and integrating over $(\la_1,\infty)$, we get
\begin{equation*}
		\underbrace{\int_{\la_1}^{\infty}\frac{\la^{a-1}}{e+\la^a}\iint_{\Phi_{\la,k}^{r_1}}(|\na u|+1)^p\ dz\ d\la}_{=:I}
		\apprle \underbrace{\int_{\la_1}^{\infty}\frac{\la^{p-q+a-1}}{e+\la^a}\iint_{\Phi_{\la,k}^{r_2}}(|\na u|+1)^q\ dz\ d\la}_{=:II}+\underbrace{\int_{\la_1}^{\infty}\frac{\la^{a-1}}{e+\la^a}\iint_{\Sig_{\la}^{r_2}}|\tf|\ dz\ d\la}_{=:III}.
\end{equation*}
\begin{description}
	\item [Estimate of $I$:] Apply Fubini to get
	\begin{equation*}
		\begin{array}{rcl}
			I
			=\iint_{\Phi_{\la_1,k}^{r_1}}(|\na u|+1)^p\int_{\la_1}^{|\na u|_k}\frac{\la^{a-1}}{e+\la^a}\ d\la\ dz
			&\ge&\frac{1}{a}\iint_{\Phi_{\la_1,k}^{r_1}}(|\na u|+1)^p\log(e+|\na u|_k^a)\ dz\\
			&& -\frac{1}{a}\iint_{\Phi_{\la_1,k}^{r_1}}(|\na u|+1)^p\log(e+\la_1^a)\ dz.
		\end{array}
	\end{equation*}
	\item [Estimate of $II$:] Fubini gives 
	\begin{equation*}
	\begin{array}{rclll}
		II =\iint_{\Phi_{\la_1,k}^{r_2}}(|\na u|+1)^q \int_{\la_1}^{|\na u|_k}\frac{\la^{p-q+a-1}}{e+\la^a}\ d\la \ dz
		&\leq& \iint_{\Phi_{\la_1,k}^{r_2}}(|\na u|+1)^q \int_{\la_1}^{|\na u|_k}{\la^{p-q-1}}\ d\la \ dz \\
		 &\leq&  \frac{1}{p-q} \iint_{\Phi_{\la_1,k}^{r_2}}(|\na u|+1)^p\ dz.
		\end{array}
	\end{equation*}

	\item [Estimate of $III$:] Fubini implies
	\begin{equation*}
		III=\iint_{\Sig_{\la_1}^{r_2}}|\tf|\int_{\la_1}^{|\tf|}\frac{\la^{a-1}}{e+\la^a}\ d\la \ dz\le \frac{1}{a}\iint_{\Sig_{\la_1}^{r_2}}|\tf|\log(e+|\tf|^a)\ dz.
	\end{equation*}
\end{description}
Combining the above estimates, we get
\begin{equation*}
		\iint_{\Phi_{\la_1,k}^{r_1}}(|\na u|+1)^p\log(e+|\na u|_k^a)\ dz\apprle_{(n,p,\La_0,\La_1,\de,a)} \log(e+\la_1^a)\iint_{\Phi_{\la_1,k}^{r_2}}(|\na u|+1)^p\ dz+\iint_{\Sig_{\la_1}^{r_2}}|\tf|\log(e+|\tf|^a)\ dz.
\end{equation*}
On the other hand, we trivially have 
\begin{equation*}
		\iint_{\{z\in Q_{r_1}(z_0):|\na u|_k\le \la_1\}}(|\na u|+1)^p\log(e+|\na u|_k^a)\ dz
		\le \log(e+\la_1^a)\iint_{\{z\in Q_{r_2}(z_0):|\na u|_k\le \la_1\}}(|\na u|+1)^p\ dz,
\end{equation*}
which gives
\begin{equation*}
	\fiint_{Q_{r_1}(z_0)}(|\na u|+1)^p\log(e+|\na u|_k^a) \ dz
	\apprle\log(e+\la_0^d)\fiint_{Q_{r_2}(z_0)}(|\na u|+1)^p\ dz+\fiint_{Q_{r_2}(z_0)}|\tf|\log(e+|\tf|)\ dz.
\end{equation*}
Take $r_1=r$ and $r_2=2r$ and let $k\to\infty$, which completes the proof.

\section{Proof of \texorpdfstring{\cref{cor_1}}.}

 Let $a >0$ and $m \in \NN$ be a given integer, then in \cref{lv_est}, we multiply with $ \frac{\la^{a-1}}{e+\la^a}\log_{(m-1)}(e+\la^a)$ and integrate to get
 \begin{description}
 	\item [Estimate of $I$:] Apply Fubini to get
 	\begin{equation*}
 	\begin{array}{rcl}
 	I
 	&=&\iint_{\Phi_{\la_1,k}^{r_1}}(|\na u|+1)^p\int_{\la_1}^{|\na u|_k}\frac{\la^{a-1}}{e+\la^a}\log_{(m-1)}(e+\la^a)\ d\la\ dz\\
 	&\ge&\frac{1}{a}\iint_{\Phi_{\la_1,k}^{r_1}}(|\na u|+1)^p\log_{(m)}(e+|\na u|_k^a)\ dz\\
 	&& -\frac{1}{a}\iint_{\Phi_{\la_1,k}^{r_1}}(|\na u|+1)^p\log_{(m)}(e+\la_1^a)\ dz.
 	\end{array}
 	\end{equation*}
 	\item [Estimate of $II$:] Fubini gives 
 	\begin{equation*}
 	\begin{array}{rcl}
 	II &=&\iint_{\Phi_{\la_1,k}^{r_2}}(|\na u|+1)^q \int_{\la_1}^{|\na u|_k}\log_{(m-1)}(e+\la^a)\frac{\la^{p-q+a-1}}{e+\la^a}\ d\la \ dz\\
 	&\leq& \iint_{\Phi_{\la_1,k}^{r_2}}(|\na u|+1)^q \log_{(m-1)}(e+|\nabla u|_k^a) \int_{\la_1}^{|\na u|_k}{\la^{p-q-1}}\ d\la \ dz \\
 	&\leq&   \frac{1}{p-q}\iint_{\Phi_{\la_1,k}^{r_2}}(|\na u|+1)^p\log_{(m-1)}(e+|\nabla u|_k^a)\ dz.
 	\end{array}
 	\end{equation*}
 	
 	\item [Estimate of $III$:] Fubini implies
 	\begin{equation*}
 	III=\iint_{\Sig_{\la_1}^{r_2}}|\tf|\int_{\la_1}^{|\tf|}\log_{(m-1)}(e+\la^a)\frac{\la^{a-1}}{e+\la^a}\ d\la \ dz\le \frac{1}{a}\iint_{\Sig_{\la_1}^{r_2}}|\tf|\log_{(m)}(e+|\tf|^a)\ dz.
 	\end{equation*}
 \end{description}
 Proceeding as in \cref{sub2.2}, we have 
\begin{equation*}
\begin{array}{rcl}
\fiint_{Q_{r}(z_0)}(|\na u|+1)^p\log_{(m)}(e+|\na u|^a) \ dz
& \apprle & \log_{(m)}(e+\la_0^d)\fiint_{Q_{2r}(z_0)}(|\na u|+1)^p\ dz \\
&&+\fiint_{Q_{2r}(z_0)}|\tf|\log_{(m)}(e+|\tf|)\ dz\\
&& + \fiint_{Q_{2r}(z_0)}(|\na u|+1)^p\log_{(m-1)}(e+|\nabla u|^a)\ dz.
\end{array}
\end{equation*}
Iterating the estimate over $m = \{0,1,\ldots,m_0\}$  proves the Corollary.

\section{Proof of \texorpdfstring{\cref{cor_2}}.}

Given $\de\in(0,1)$ and $\mfa = \frac{n \de p}{n-\de p}$,  then there exists $\ve_1=\ve_1(p,\de)$ such that for any $\ve\in(0,\ve_1)$, there holds
\begin{equation*}
\frac{p}{p-1}+\frac{\ve}{p-1}<(\mfa')^* \quad \Longleftrightarrow \quad \ve < \frac{p(1-\de)}{p\de -1}.
\end{equation*}
Applying \cref{lem_1.34} for a.e. $t\in I_{2r}(t_0)$, it follows that
\begin{equation*}
\fint_{B_{2r}(x_0)} |I_1 g(x,t)|^{\frac{p}{p-1}+\frac{\ve}{p-1}}\ dx\apprle_{(n,p,\de)} (2r)^{\alpha}\lbr\fint_{B_{2r}(x_0)}(2r)^{\mfa'}|g(x,t)|^{\mfa'}\ dx\rbr^{\frac{p}{(p-1)\mfa'}+\frac{\ve}{(p-1)\mfa'}},
\end{equation*}
where
\begin{equation*}
\alpha=\lbr \frac{p+\ve}{p-1}\rbr \lbr \frac{n\ga}{s} + \frac{n}{\mfa'} - 1\rbr - n \quad \text{with} \ \ga = \frac{(1-n)s+n}{s} \ \text{and} \ s= \frac{1}{1-\frac{1}{\mfa'}+\frac{p-1}{p+\ve}}.
\end{equation*}
With these choices, it is easy to see that $\al =0$, which immediately gives
\begin{equation}\label{est3_1_c}
\fiint_{Q_{2r}(z_0)}|I_1g|^{\frac{p}{p-1}+\frac{\ve}{p-1}}\ dx\apprle \fint_{I_{2r}(t_0)}\lbr \fint_{B_{2r}(x_0)}(2r)^{\mfa'}|g|^{\mfa'}\ dx\rbr^{\frac{p}{(p-1)\mfa'}+\frac{\ve}{(p-1)\mfa'}}\ dt.
\end{equation}
On the other hand, \cref{thm_2} gives that there exists a $\ve_2=\ve_2(n,p,\La_0,\La_1)$ such that for any $\ve\in(0,\ve_2)$ there holds
\begin{equation}\label{est3_2_c}
\fiint_{Q_r(z_0)}|\na u|^{p+\ve}\ dz\apprle_{(n,p,\La_0,\La_1)}\la_0^{\ve}\fiint_{Q_{2r}(z_0)}(|\na u|+1)^p\ dz+\fiint_{Q_{2r}(z_0)}|I_1g|^{\frac{p}{p-1}+\frac{\ve}{p-1}}\ dx.
\end{equation}

Hence, taking $\ve_0=\min\{\ve_1,\ve_2\}$ and combining \cref{est3_1_c} and \cref{est3_2_c}, this completes the proof.

\section{Proof of \texorpdfstring{\cref{thm_3}}. - Borderline gradient potential estimates}
\label{section7}
In this section, we consider a weak solution $u\in C(0,T;L^2(\Omega))\cap L^p(0,T;W^{1,p}(\Omega))$ of
\begin{equation}\label{p_main}
u_t-\dv|\na u|^{p-2}\na u=f(x)\text{ in }\Omega_T.
\end{equation}

\subsection{Fixing some quantities}

\begin{definition}\label{def5.1}
	Let us define the following constants:
	\begin{description}
		\descitem{C7}{C7} We denote $\mfa:=\frac{np}{n-p}$ if $p<n$ and $\mfa:=\infty$ if $p>n$.
		\descitem{C8}{C8} For $p\ne n$ and $\mfa$ defined above, we denote
		\begin{gather*}
		\bM_p^f(B_r(x_0)):=\int_0^r\lbr[][\frac{|f|^{\mfa'}(B_s(x_0))}{s^{n-\mfa'}}\rbr[]]^{\frac{1}{\mfa'(p-1)}}\ \frac{ds}{s}\txt{and}\bP^f_{\mfa'}(B_r(x_0)):=\int_0^r\lbr\frac{|f|^{\mfa'}(B_s(x_0))}{s^{n-\mfa'}}\rbr^\frac{1}{\mfa'}\ \frac{ds}{s}.
		\end{gather*}
		\descitem{C9}{C9}  In the case $p=n$, we shall denote
		\begin{equation*}
		| f|_{L\log L(B_s(x_0))}:=\int_{B_s(x_0)}|f|\log\lbr e+\frac{|f|}{\int_{B_{s}(x_0)}|f|\ dx}\rbr\ dx,
		\end{equation*}
		and
		\begin{equation*}
		\bM_n^{f}(B_r(x_0)):=\int_0^r\lbr[[] s^{-(n-1)}| f|_{L\log L(B_s(x_0))}\rbr[]]^\frac{1}{n-1}\ \frac{ds}{s},
		\end{equation*}
		where the norm $L\log L$ is defined in \cref{llog_form}.
		\descitem{C10}{C10} We define
		\begin{equation*}
		\bF=\bF(B_{2\rho}):=
		\begin{cases}
		\bP_{\mfa'}^{f}(B_{2\rho})&\text{ if }p< 2,\\
		\bM_p^f(B_{2\rho})&\text{ if }p\ge 2.
		\end{cases}
		\end{equation*}
	\end{description}
\end{definition}

\begin{definition}
	Let $q\in[1,\infty)$, then for any measurable set $Q$ and integrable function $h$ defined in $Q$, we define excess functional 
	\begin{equation*}
	E_q(h,Q):=\lbr\fiint_{Q}|h-(h)_{Q}|^q\ dz\rbr^\frac{1}{q}.
	\end{equation*}
\end{definition}
\begin{remark}
	Note that for any $\Gamma\in \RR^n$, there holds that
	\begin{equation}\label{tw}
	E_q(h,Q)\le 2\lbr\fiint_{Q}|h-\Gamma|^q\ dz\rbr^\frac{1}{q}.
	\end{equation}
	\end{remark}

The proof of the following lemma can be found in \cite[Equation (85)]{MR2729305}.
\begin{lemma}\label{p_sum}
	Let $p\ne n$, $\sig\in(0,1)$ and $r_j:=\sig^{j}r$. Then there holds that
	\begin{gather*}
	\sum_{1\le j<\infty}\lbr[][\frac{|f|^{\mfa'}(B_{r_j}(x_0))}{r_j^{n-\mfa'}}\rbr[]]^{\frac{1}{\mfa'(p-1)}}\apprle_{(n,p,\sig)}\bM_p^f(B_r(x_0))\txt{and}\sum_{1\le j<\infty}\lbr[[]\frac{|f|^{\mfa'}(B_{r_j}(x_0))}{r_{j}^{n-\mfa'}}\rbr[]]^\frac{1}{\mfa'}\apprle_{(n,\sig)}\bP_{\mfa'}^f(B_r(x_0)).
	\end{gather*}
\end{lemma}
Similar calculations as in the proof of \cref{p_sum} give the following lemma:
\begin{lemma}\label{log_sum}
	Let $p=n$, $\sig\in(0,1)$ and $r_j:=\sig^{j}r$. Then there holds that
	\begin{equation*}
	\sum_{1\le j<\infty}\lbr[][r_j^{-(n-1)}| f |_{L\log L(B_{r_j}(x_0))}\rbr[]]^\frac{1}{n-1}\apprle_{(\sig)}\int_0^r \lbr[[]s^{-(n-1)}| f |_{L\log L(B_s(x_0))}\rbr[]]^\frac{1}{n-1}\ \frac{ds}{s}.
	\end{equation*}
\end{lemma}
\begin{proof}
	For $r \in [r_j,r_{j-1}]$, we have
	\begin{equation*}
	\begin{array}{rcl}
	|f|_{L\log L(B_{r_j}(x_0))}
	&=&\int_{B_{r_j}(x_0)}|f|\log\lbr \frac{e+\frac{|f|}{\int_{B_{r_j}(x_0)}|f|\ dx}}{e+\frac{|f|}{\int_{B_{r}(x_0)}|f|\ dx}} \lbr e+\frac{|f|}{\int_{B_{r}(x_0)}|f|\ dx}\rbr\rbr\ dx\\
	&\le&\int_{B_{r_j}(x_0)}|f|\log\lbr e+\frac{|f|}{\int_{B_{r}(x_0)}|f|\ dx}\rbr+|f|\log\lbr 1+\frac{\int_{B_{r}(x_0)}|f|\ dx}{\int_{B_{r_j}(x_0)}|f|\ dx}\rbr\ dx,
	\end{array}
	\end{equation*}
	where to obtain the last inequality, we made use of the bound $\frac{e+\frac{|f|}{\int_{B_{r_j}(x_0)}|f|\ dx}}{e+\frac{|f|}{\int_{B_{r}(x_0)}|f|\ dx}}  \leq 1 + \frac{\int_{B_{r}(x_0)}|f|\ dx}{\int_{B_{r_j}(x_0)}|f|\ dx}$.

	Moreover, we see that
	\begin{equation*}
	\begin{array}{rcl}
	\int_{B_{r_j}(x_0)}|f|\log\lbr 1+\frac{\int_{B_{r}(x_0)}|f|\ dx}{\int_{B_{r_j}(x_0)}|f|\ dx}\rbr\ dx
	&\le&\frac{\int_{B_{r}(x_0)}|f|\ dx}{\int_{B_{r_j}(x_0)}|f|\ dx}\int_{B_{r_j}(x_0)}|f|\ dx\\
	&\le&\int_{B_{r}(x_0)}|f|\ dx\\
	&\le& |f|_{L\log L(B_r(x_0))}.
	\end{array}
	\end{equation*}
	Therefore we obtain 
	\begin{equation*}
	\lbr[][r_j^{-(n-1)}| f |_{L\log L(B_{r_j}(x_0))}\rbr[]]^\frac{1}{n-1}
	\le\frac{\log \tfrac{1}{\sig}}{\sig}\int_{r_j}^{r_{j-1}}\lbr[][\frac{| f |_{L\log L(B_r(x_0))}}{r^{n-1}}\rbr[]]^\frac{1}{n-1}\frac{dr}{r}.
	\end{equation*}
\end{proof}

\begin{definition}\label{def7.4}
	Let $\rho$ be fixed and for $Q_{2\rho}\subset\Om_T$, we assume that there exist $\bH_1,\bH_2$ and $\la\ge1$ such that
	\begin{equation}\label{p_la}
	\bH_1^p\fiint_{Q_{2\rho}^\la}|\na u|^p+1\ dz+\bH_2^{p}\mathbf{F}^{\max\{\frac{p}{p-1},p\}}=\la^p.
	\end{equation}
	These constants will eventually be chosen appropriately in \cref{decay_est} and \cref{final_pf}.
\end{definition}

We will need the following simple estimate:
\begin{lemma}\label{p_cancel}
	Suppose $p< 2$ and \cref{p_la} hold. Let $z_0\in Q_{2r}^\la (z_0)\subset Q_{2\rho}$, then we have
	\begin{equation*}
\lbr\fint_{B_{r}(x_0)}(r|f|)^{\mfa'}\ dx\rbr^\frac{1}{\mfa'} \apprle_{(n,p)}\lbr\frac{\la}{\bH_2}\rbr^{p-1}.
	\end{equation*}
\end{lemma}
\begin{proof}
	There exists $k\in\NN$ such that
	$2^{-(k+1)}\rho\le r\le 2^{-k}\rho$. Applying \cref{p_sum} to get
	\begin{equation*}
	\lbr\fint_{B_{r}(x_0)}(r|f|)^{\mfa'}\ dx\rbr^\frac{1}{\mfa'}\apprle_{(n)} \lbr\fint_{B_{2^{-k}\rho}(x_0)}((2^{-k}\rho)|f|)^{\mfa'}\ dx\rbr^\frac{1}{\mfa'}\overset{\text{\cref{p_sum}}}{\apprle}\bF\overset{\cref{p_la}}{\le}\lbr\frac{\la}{\bH_2}\rbr^{p-1}.
	\end{equation*}
	This completes the proof.
\end{proof}

\subsection{Regularity for the homogeneous equation}
Let us enumerate Lemmas will be used in this section. First of all, fundamental regularity results obtained by DiBenedetto in \cite[Chapter VIII-IX]{MR1230384}. More specifically, the Lipschitz estimate is from \cite[Theorems 5.1 and 5.2 of Chapter VIII]{MR1230384} and the oscillation decay estimate can be found in \cite[Section 3 of Chapter IX]{MR1230384}.
\begin{lemma}\label{hom_up}
	Let $\la\ge1$ and $v\in L^p(I_r^\la(t_0),W^{1,p}(B_r(x_0)))$ be a weak solution of 
	\begin{equation*}
	v_t-\dv|\na v|^{p-2}\na v=0\text{ in }Q_r^\la(z_0).
	\end{equation*}
	Suppose there exists $c_*\ge1$ such that
	\begin{equation*}
	\fiint_{Q_r^\la(z_0)}|\na v|^p\ dz\le c_*\la^p.
	\end{equation*}
	Then there holds that
	\begin{gather*}
	\sup_{\frac{1}{2}Q_r^\la(z_0)}|\na v|< c\la\txt{and}\osc\limits_{sQ_r^\la(z_0)}\iprod{\nabla v}{e_i}\leq cs^\alpha\la,
	\end{gather*}
	where $e_i = (0,\ldots,0, \underbrace{1}_{i^{th} place}, 0,\ldots,0)$ with $1 \leq i \leq n$, $c=c(n,p,c_*)$, $\alpha=\alpha(n,p,c_*)\in(0,1)$ and $s\in(0,\tfrac{1}{2})$.  
\end{lemma}

The proof of the following decay estimate can be found in \cite[Theorem 3.1]{MR2968162}.
\begin{lemma}\label{km_alt}
	Let $\la\ge1$ and $v\in L^p(I_r^\la(t_0),W^{1,p}(B_r(x_0)))$ be a weak solution of 
	\begin{equation*}
	v_t-\dv|\na v|^{p-2}\na v=0\text{ in }Q_r^\la(z_0).
	\end{equation*}
	For any $q\in[1,\infty)$, consider $A,B\ge1$ and $\gamma\in(0,1)$. Then there exists $\sig_0=\sig_0(n,p,A,B,\gamma)$ such that if
	\begin{equation}\label{Dib_alt}
	\frac{\la}{B}\le \sup_{\sig_0 Q_r^\la(z_0)}|\na v|\le \sup_{Q_{r}^\la(z_0)}|\na v|\le A\la,
	\end{equation}
	holds, then for any $\sig\in(0,\sig_0)$, the following holds:
	\begin{equation*}
	E_q(\na v,\sig Q_r^\la(z_0))\le \gamma E_q(\na v,Q_r^\la).
	\end{equation*}
\end{lemma}

\subsection{Comparison estimate when \texorpdfstring{$p \neq n$}.}
Consider $v\in C(I_r^\la(t_0);W^{1,p}(B_r(x_0)))$ be a weak solution of 
\begin{equation}\label{p_hom}
\begin{cases}
v_t-\dv |\na v|^{p-2}\na v=0&\text{ in }Q_r^\la(z_0),\\
v=u&\text{ on }\pa_pQ_r^\la(z_0).
\end{cases}
\end{equation}
Then we have the following estimate:
\begin{lemma}\label{diff}
	Let $u$ be a weak solution of \cref{p_main} and  $Q_r^\la(z_0)\subset\Omega_T$ and \cref{p_la} hold. For weak solution $v$ of \cref{p_hom} if $p\ge2$ and $p\ne n$, then there holds
	\begin{equation}\label{7.6}
	\fiint_{Q_r^\la(z_0)}|\na u-\na v|^p\ dz\apprle_{(n,p)} \lbr\fint_{B_r(x_0)}r^{\mfa'}|f|^{\mfa'}\ dx\rbr^{\frac{1}{\mfa'}\frac{p}{p-1}},
	\end{equation}
	and if $p<2$, then there holds
	\begin{equation}\label{7.7}
	\fiint_{Q_r^\la(z_0)}|\na u-\na v|^p\ dz
	\apprle_{(n,p)} \lbr\fint_{B_r^{\la}(x_0)}r^{\mfa'}|f|^{\mfa'}\ dx\rbr^{\frac{1}{\mfa'}\frac{p}{p-1}}+\lbr\frac{\la}{\bH_1}\rbr^{(2-p)p}\lbr\fint_{B_r^{\la}(x_0)}r^{\mfa'}|f|^{\mfa'}\ dx\rbr^{\frac{p}{\mfa'}}.
	\end{equation}
\end{lemma}
\begin{proof}
	The proof is analogue to \cite[Lemma 5]{MR2729305} for the case $p\ge 2$ and $p\ne n$ and to \cite[Lemma 3.3]{MR3247381} for the case $p<2$. For the sake of completeness, we present all the details.

We make use of  $u-v$ as a test function for $(u-v)_t-\dv\lbr|\na u|^{p-2}\na u-|\na v|^{p-2}\na v\rbr=g$. Then, we obtain
	\begin{equation*}
	\begin{array}{rcl}
\fiint_{Q_r^\la(z_0)} \pa_t (u-v)^2 \ dz + 	 \underbrace{\fiint_{Q_r^\la(z_0)}\iprod{|\na u|^{p-2}\na u-|\na v|^{p-2}\na v}{\na u-\na v}\ dz}_{=:I}
	\le \underbrace{\fiint_{Q_r^\la(z_0)}|f||u-v|\ dz}_{=:II}.
	\end{array}
	\end{equation*}
	Note that $u-v=0$ at the bottom of the cylinder $Q_r^{\la}(z_0)$ and hence we can ignore the contribution arising from this term. 
	\begin{description}
		\item[Estimate of $I$:]
		Applying \cref{p_str}, we get
		\begin{equation*}
		\fiint_{Q_r^\la(z_0)}\lbr|\na u|^2+|\na v|^2\rbr^\frac{p-2}{2}|\na u-\na v|^2\ dz\apprle I.
		\end{equation*}
		\item[Estimate of $II$:]
		 To estimate this term, we further consider the following subcases:
		   \begin{description}
		   	\item[Case $1 < p < n$:] In this case, we have
		   		\begin{equation*}
		   	\begin{array}{rcl}
		   	II
		   	&\le&\fint_{I_r^\la(t_0)}\lbr\fint_{B_r(x_0)}r^{\mfa'}|f|^{\mfa'}\ dx\rbr^\frac{1}{\mfa'}\lbr\fint_{B_r(x_0)}\abs{\frac{u-v}{r}}^\mfa\ dx\rbr^\frac{1}{\mfa}\ dt\\
		   	&\overset{\redlabel{7.5a}{a}}{\apprle}&\fint_{I_r^\la(t_0)}\lbr\fint_{B_r(x_0)}r^{\mfa'}|f|^{\mfa'}\ dx\rbr^\frac{1}{\mfa'}\lbr\fint_{B_r(x_0)} |\na u-\na v|^p\ dx\rbr^\frac{1}{p}\ dt\\
		   	&\overset{\redlabel{7.5b}{b}}{\le}&C(\gamma)\lbr\fint_{B_r(x_0)}r^{\mfa'}|f|^{\mfa'}\ dx\rbr^{\frac{1}{\mfa'}\frac{p}{p-1}}+\gamma\fiint_{Q_r^\la(z_0)} |\na u-\na v|^p\ dz,
		   	\end{array}
		   	\end{equation*}
		   	where to obtain \redref{7.5a}{a}, we used Poincar$\acute{\text{e}}$ inequality and to obtain \redref{7.5b}{b}, we used Young's inequality.

		   	\item[Case $p>n$:] In this case, we have
		   		\begin{equation*}
		   	\begin{array}{rcl}
		   	II
		   	&\leq& \lbr \fint_{B_{r}(x_0)}|f|\ dx\rbr \fint_{I_r^\la(t_0)}\sup_{B_r(x_0)}|u(\cdot,t)-v(\cdot,t)|\ dt\\
		   	&\overset{\redlabel{7.7a}{a}}{\apprle}&\lbr \fint_{B_{r}(x_0)}r|f|\ dx\rbr \fint_{I_r^\la(t_0)}\lbr\fint_{B_r(x_0)}|\na u-\na v|^p\ dx\rbr^\frac{1}{p} \ dt\\
		   	&\overset{\redlabel{7.7b}{b}}{\le}& C(\gamma)\lbr\fint_{B_{r}(x_0)}r|f|\ dx\rbr^\frac{p}{p-1}+\gamma\fiint_{Q_r^\la(z_0)}|\na u-\na v|^p\ dz,
		   	\end{array}
		   	\end{equation*}
		   	where to obtain \redref{7.7a}{a}, we used \cref{Morrey} and to obtain \redref{7.7b}{b}, we used Young's inequality.
		   \end{description}
	\end{description}
   We now combine the previous estimates to prove the lemma as follows:
   \begin{description}
   	\item[Proof of \cref{7.6}:] Applying \cref{p_str}, we obtain
   	\begin{equation*}
   	\fiint_{Q_r^\la(z_0)}|\na u-\na v|^p\ dz \apprle  C(\gamma)\lbr\fint_{B_r(x_0)}r^{\mfa'}|f|^{\mfa'}\ dx\rbr^{\frac{1}{\mfa'}\frac{p}{p-1}}+\gamma\fiint_{Q_r^\la(z_0)}|\na u-\na v|^p\ dz.
   	\end{equation*}
   	Choosing $\ga$ suitably small gives the desired estimate.
   	   	\item[Proof of \cref{7.7}:] \cref{p_str} implies 
   	   	\begin{equation*}
   	   		\begin{array}{rcl}
   	   			\fiint_{Q_r^\la(z_0)}|\na u-\na v|^p\ dz
   	   			&&\apprle\fiint_{Q_r^\la(z_0)}\lbr|\na u|^2+|\na v|^2\rbr^\frac{p-2}{2}|\na u-\na v|^2\ dz\\
   	   			&&\quad+\lbr\fiint_{Q_r^\la(z_0)}\lbr|\na u|^2+|\na v|^2\rbr^\frac{p-2}{2}|\na u-\na v|^2\ dz\rbr^\frac{p}{2}\lbr\fiint_{Q_r^\la(z_0)}|\na u|^p\ dz\rbr^\frac{2-p}{2}\\
   	   			&&\overset{\cref{p_la}}{\le}\fiint_{Q_r^\la(z_0)}\lbr|\na u|^2+|\na v|^2\rbr^\frac{p-2}{2}|\na u-\na v|^2\ dz\\
   	   			&&\qquad+\lbr\fiint_{Q_r^\la(z_0)}\lbr|\na u|^2+|\na v|^2\rbr^\frac{p-2}{2}|\na u-\na v|^2\ dz\rbr^\frac{p}{2}\lbr\frac{\la}{\bH_1}\rbr^\frac{(2-p)p}{2},
   	   		\end{array}
   	   	\end{equation*}
      	and thus, Young's inequality gives
   	   	\begin{equation}\label{eq7.16}
   	   	\begin{array}{rcl}
   	   	\fiint_{Q_r^\la(z_0)}|\na u-\na v|^p\ dz 
   	   	& \leq & C(\gamma)\lbr\fint_{B_{r}(x_0)}r|f|\ dx\rbr^\frac{p}{p-1}+\gamma\fiint_{Q_r^\la(z_0)}|\na u-\na v|^p\ dz \\
   	   	&& + \lbr\fiint_{Q_r^\la(z_0)}|f||u-v|\ dz \rbr^{\frac{p}{2}}\lbr\frac{\la}{\bH_1}\rbr^\frac{(2-p)p}{2}.
   	   	\end{array}
   	   	\end{equation}
   	   	In order to estimate the last term, we proceed as follows:
   	   		\begin{multline}
   	   	\lbr\fiint_{Q_r^\la(z_0)}|f||u-v|\ dz\rbr^\frac{p}{2}\lbr\frac{\la}{\bH_1}\rbr^\frac{(2-p)p}{2}\\
   	   	\le\lbr\fint_{B_r(x_0)}r^{\mfa'}|f|^{\mfa'}\ dx\rbr^{\frac{1}{\mfa'}\frac{p}{2}}\lbr\fiint_{Q_r^\la(z_0)} |\na u-\na v|^p\ dz\rbr^\frac{1}{2}\lbr\frac{\la}{\bH_1}\rbr^\frac{(2-p)p}{2}\\
   	  	\overset{\redlabel{7.6a}{a}}{\le}C(\gamma)\lbr\frac{\la}{\bH_1}\rbr^{(2-p)p}\lbr\fint_{B_r(x_0)}r^{\mfa'}|f|^{\mfa'}\ dx\rbr^{\frac{p}{\mfa'}}+\gamma\fiint_{Q_r^\la(z_0)} |\na u-\na v|^p\ dz. \label{diff_p_le_n}
   	   	\end{multline}
   	   	Combining \cref{eq7.16} and \cref{diff_p_le_n} followed by choosing $\ga$ small gives the desired estimate.
   \end{description}
	This completes the proof of the lemma.
\end{proof}
 
\begin{corollary}\label{p_le_2} 
	In the case $1< p \leq 2$ under the hypothesis of \cref{diff}, we have
	\begin{equation*}
	\lbr\fiint_{Q_r^\la(z_0)}|\na u-\na v|^p\ dz\rbr^\frac{1}{p}\apprle \lbr\frac{\la}{\bH_1}\rbr^{2-p}\lbr\fint_{B_r(x_0)}\lbr r|f|\rbr^{\mfa'}\ dx\rbr^\frac{1}{\mfa'}.
	\end{equation*}
\end{corollary}
\begin{proof}
	Let us estimate the first term in \cref{7.7}. Since $p \leq 2$, we have $p \leq p'$ which gives
	\begin{equation*}
	\begin{array}{rcl}
	\lbr\fint_{B_r(x_0)}r^{\mfa'}|f|^{\mfa'}\ dx\rbr^{\frac{1}{\mfa'}\frac{p}{p-1}}
	&=&\lbr\fint_{B_r(x_0)}r^{\mfa'}|f|^{\mfa'}\ dx\rbr^{\frac{p}{\mfa'}}\lbr\fint_{B_r(x_0)}r^{\mfa'}|f|^{\mfa'}\ dx\rbr^{\frac{1}{\mfa'}\frac{(2-p)p}{p-1}}\\
	&\overset{\redlabel{7.24a}{a}}{\le}&\lbr\frac{\la}{\bH_1}\rbr^{(2-p)p}\lbr\fint_{B_r(x_0)}r^{\mfa'}|f|^{\mfa'}\ dx\rbr^{\frac{p}{\mfa'}},
	\end{array}
	\end{equation*}
	where to obtain \redref{7.24a}{a}, we used \cref{p_cancel}. This completes the proof of the corollary.
\end{proof}
\subsection{Comparison estimate when \texorpdfstring{$p = n$}.}
Consider the weak solution $v\in C(I_r^\la(t_0);W^{1,p}(B_r(x_0)))$ of 
\begin{equation}\label{p_hom_n}
\begin{cases}
v_t-\dv |\na v|^{p-2}\na v=0&\text{ in }Q_r^\la(z_0),\\
v=u&\text{ on }\pa_pQ_r^\la(z_0).
\end{cases}
\end{equation}
\begin{lemma}\label{diff_n}
	Let $u$ be a weak solution of \cref{p_main}, $Q_r^\la(z_0)\subset\Omega_T$ and \cref{p_la} hold. Suppose that $p= n$, then there exists $r_0=r_0(n)$ such that for any $r\in(0,r_0)$ and for weak solution $v$ of \cref{p_hom_n}, the following holds
	\begin{equation*}
	\lbr\fiint_{Q_r^\la(z_0)}|\na u-\na v|^n\ dz\rbr^\frac{1}{n}\apprle_{(n)}\lbr\fint_{B_r(x_0)}r|f|\log\lbr e+\frac{|f|}{\fint_{B_{r}(x_0)}|f|\ dx}\rbr\ dx\rbr^\frac{1}{n-1}.
	\end{equation*}
\end{lemma}
\begin{proof}
	We will prove the lemma at $r_0$ which will be determined according to \cref{diff_n_radi}. For $(x,t)\in Q_{r_0}$, we define 
	\begin{equation*}
	\tu(x,t):=\la^{-1} u\lbr x_0+x,t_0+\la^{2-p}t\rbr \txt{and} \tf(x):=\la^{1-p}f\lbr x_0+x
	\rbr.
	\end{equation*}
	Then $\tu\in C(I_{r_0};L^2(B_{r_0}(x_0)))\cap L^p(I_{r_0};W^{1,p}(B_{r_0}))$ is a weak solution of 
	\begin{equation*}
	\tu_t-\dv|\na \tu|^{p-2}\na \tu=\tf\text{ in }Q_{r_0}.
	\end{equation*}
	Let $\tv\in C(I_{r_0};L^2(B_{r_0}))\cap L^{n}(I_{r_0};W^{1,n}(B_{r_0}))$ be the weak solution of
	\begin{equation*}
	\begin{cases}
	\tv_t-\dv |\na \tv|^{n-2}\na \tv=0&\text{ in }Q_{r_0},\\
	\tv=\tu&\text{ on }\partial_pQ_{r_0}.
	\end{cases}
	\end{equation*}
	Using $\tu-\tv$ as a test function for  $(\tu-\tv)_t-\dv\lbr|\na \tu|^{p-2}\na \tu-|\na \tv|^{p-2}\na \tv\rbr=\tf$, the standard energy estimate gives the following inequality
	\begin{equation*}
\fiint_{Q_r^\la(z_0)} \pa_t (\tu-\tv)^2 \ dz + 	 \underbrace{\fiint_{Q_r^\la(z_0)}\iprod{|\na \tu|^{p-2}\na \tu-|\na \tv|^{p-2}\na \tv}{\na \tu-\na \tv}\ dz}_{=:I}
\le \underbrace{\fiint_{Q_r^\la(z_0)}|\tf||\tu-\tv|\ dz}_{=:II}.
	\end{equation*}
	\begin{description}[leftmargin=12pt]
		\item [Estimate of $I$:] Applying \cref{p_str}, we get
		\begin{equation*}
		\fiint_{Q_{r_0}}|\na \tu-\na \tv|^n\ dz\apprle\fiint_{Q_{r_0}}\lbr|\na \tu|^2+|\na \tv|^2\rbr^\frac{n-2}{2}|\na \tu-\na \tv|^2\ dz\apprle_{(n)}II.
		\end{equation*}
		\item [Estimate of $II$:]  We will split this integral into two parts as follow:
		\begin{equation*}
		\begin{array}{rcl}
		II\le \fiint_{Q_{r_0}}|\tf||\tu-\tv-\avgs{\tu-\tv}{B_{r_0}}|\ dz+\fiint_{Q_{r_0}}|\tf||\avgs{\tu-\tv}{B_{r_0}}|\ dz:=II_1+II_2.
		\end{array}
		\end{equation*}
		
		\begin{description}[leftmargin=0pt]
			\item [Estimate of $II_1$:]
			Applying \cref{orl_holder} with $\Phi(s)=(1+s)\log(1+s)-s$ and $\tilde{\Phi}(s)=\exp(s)-s-1$, we get 
			\begin{equation*}
			II_1\apprle \fint_{I_{r_0}}\frac{1}{|B_{r_0}|}\lVert \tf\rVert_{L^\Phi(B_{r_0})} \lVert \tu-\tv-\avgs{\tu-\tv}{B_{r_0}}\rVert_{L^{\tilde{\Phi}}(B_{r_0})}\ dt.
			\end{equation*}
			Note that for all $s>0$, we have 
			\begin{equation*}
			\begin{array}{rcl}
			\int_{B_{r_0}}\tilde{\Phi}\lbr\frac{|\tu-\tv-\avgs{\tu-\tv}{B_{r_0}}|}{s}\rbr\ dx
			&\le&\int_{B_{r_0}}\exp\lbr\frac{|\tu-\tv-\avgs{\tu-\tv}{B_{r_0}}|}{s}\rbr\ dx,
			\end{array}
			\end{equation*}
			and $(\tu-\tv)(\cdot,t)\in W^{1,n}_0(B_{r_0})$ satisfies the hypothesis of \cref{J_N} with $M=\lbr\int_{B_{r_0}}|\na \tu-\na \tv|^n\ dx\rbr^\frac{1}{n}$.  Thus, taking $\sig=\sig(n)=\frac{1}{2^n}w_n\sig_0$ in  \cref{J_N}, we get
			\begin{equation}\label{diff_n_radi}
			\int_{B_{r_0}}\exp\lbr\frac{\sig|\tu-\tv-\avgs{\tu-\tv}{B_{r_0}}|}{\lbr\int_{B_{r_0}}|\na \tu-\na \tv|^n\ dx\rbr^\frac{1}{n}}\rbr\ dx\le c(n)r_0^n.
			\end{equation}
			Now, we choose  $r_0=r_0(n)$ small such that $c(n)r_0^n\le 1$ where $c(n)$ is the  constant in \cref{diff_n_radi}. This implies
			\begin{equation*}
			\frac{1}{\sig}\lbr\int_{B_{r_0}}|\na \tu-\na \tv|^n\ dx\rbr^\frac{1}{n}\le s\Longrightarrow\int_{B_{r_0}}\tilde{\Phi}\lbr\frac{|\tu-\tv-\avgs{\tu-\tv}{B_{r_0}}|}{s}\rbr\ dx\le 1,
			\end{equation*}
			and thus we have
			\begin{equation*}
			\lVert \tu-\tv-\avgs{\tu-\tv}{B_{r_0}}\rVert_{L^{\tilde{\Phi}}(B_{r_0})}\le \frac{1}{\sig}\lbr\int_{B_{r_0}}|\na \tu-\na \tv|^n\ dx\rbr^\frac{1}{n}.
			\end{equation*}
			Therefore, applying Young's inequality, we get
			\begin{equation*}
			\begin{array}{rcl}
			II_1
			&\apprle& r_0^{-(n-1)}\lVert \tf\rVert_{L^\Phi(B_{r_0})}\fint_{I_{r_0}}\lbr\fint_{B_{r_0}}|\na \tu-\na \tv|^n\ dx\rbr^\frac{1}{n}\ dt\\
			&\le&C(\gamma)\lbr r_0^{-(n-1)}\lVert \tf\rVert_{L^\Phi(B_{r_0})}\rbr^\frac{n}{n-1}+\gamma\fiint_{Q_{r_0}}|\na \tu-\na \tv|^n\ dz.
			\end{array}
			\end{equation*}
			\item [Estimate of $II_2$:] Applying Poincar$\acute{\text{e}}$ inequality in the spatial direction followed by Young's inequality, we get
			\begin{equation*}
			\begin{array}{rcl}
			II_2
			&=&\fint_{B_{r_0}}|\tf|\ dx\fint_{I_{r_0}}|(\tu-\tv)_{B_{r_0}}|\ dt\\
			&\apprle&\fint_{B_{r_0}}r_0|\tf|\ dx\lbr\fiint_{Q_{r_0}}|\na \tu-\na \tv|^n\ dz\rbr^\frac{1}{n}\\
			&\le&C(\gamma)\lbr\fint_{B_{r_0}}r_0|\tf|\ dx\rbr^\frac{n}{n-1}+\gamma\fiint_{Q_{r_0}}|\na \tu-\na \tv|^n\ dz.
			\end{array}
			\end{equation*}
		\end{description}
	\end{description}
	Therefore taking $\gamma=\gamma(n)\in(0,1)$ small enough, we obtain 
	\begin{equation*}
	\fiint_{Q_{r_0}}|\na \tu-\na \tv|^n\ dz\apprle\lbr r_0^{-(n-1)}\lVert \tf\rVert_{L^\Phi(B_{r_0})}\rbr^\frac{n}{n-1}+\lbr\fint_{B_{r_0}}r_0|\tf|\ dx\rbr^\frac{n}{n-1}.
	\end{equation*}
	And we apply \cref{l1_llog} to get
	\begin{equation*}
	\fiint_{Q_{r_0}}|\na \tu-\na \tv|^n\ dz\apprle\lbr\fint_{B_{r_0}}r_0|\tf|\log\lbr e+\frac{|\tf|}{\lVert \tf\rVert_{L^1(B_{r_0})}}\rbr \ dx\rbr^\frac{n}{n-1}.
	\end{equation*}
	Scaling back from $|\na \tu-\na \tv|$ to $\na u - \na v|$ gives the required estimate and this proves the lemma.
\end{proof}
\subsection{Decay estimate}\label{decay_est}

Let us first fix some notation:
	\begin{itemize}
	\item Let $z_0\in Q_{\rho}$ be a Lebesgue point of $\na u$ and $\sum_{j=1}^{\infty}\bF_j(z_0) < \infty$,
	\item $\sig=\sig(n,p)\in(0,1/4)$ will be determined in \cref{pre_decay},
	\item $\bH_1$  depending on $\sig,n$ and $p$ will be determined according to \cref{km},
	\item $\bH_2$ depending on $\sig,n$ and $p$ will be determined according to \cref{fix_h_2} and \cref{decay},
	\item For $j\ge 1$, we denote $\rho_j:=\lbr\frac{\sig}{2}\rbr^{j-1}\rho$,  $Q_j=Q_{\rho_j}^\la(z_0)$ and $B_j=B_{\rho_j}(x_0)$,
	
	\item We take
	\begin{equation*}
	\bF_j(z_0) = \bF_j:=
	\begin{cases}
	\lbr[][\frac{|f|^{\mfa'}(B_{\rho_j}(x_0))}{\rho_j^{n-\mfa'}}\rbr[]]^\frac{1}{\mfa'}&\text{ if }p<2,\\
	\lbr[][\frac{|f|^{\mfa'}(B_{\rho_j}(x_0))}{\rho_j^{n-\mfa'}}\rbr[]]^\frac{1}{\mfa'(p-1)}&\text{ if }p\ge2\text{ and }p\ne n,\\
	\lbr[][\rho_j^{-(n-1)}| f|_{L\log L(B_{\rho_j}(x_0))}\rbr[]]^\frac{1}{n-1}&\text{ if }p=n.	
	\end{cases}
	\end{equation*}
\end{itemize}
Let us define
\begin{equation*}
G_j:=\lbr\fiint_{Q_j}|\na u|^p\ dz\rbr^\frac{1}{p}+\lbr\frac{\sig}{2}\rbr^{-\frac{n+2}{p}}\eduj.
\end{equation*}
From \cref{p_la} and \cref{tw}, we see that
\begin{equation*}
G_1\le 3\lbr\frac{\sig}{2}\rbr^{-\frac{n+2}{p}}\lbr\fiint_{Q_1}|\na u|^p \ dz\rbr^\frac{1}{p}\le 3\lbr\frac{\sig}{2}\rbr^{-\frac{n+2}{p}}2^\frac{n+2}{p}\lbr\fiint_{Q_{2\rho}^\la}|\na u|^p\ dz\rbr^\frac{1}{p}\le 3\lbr\frac{\sig}{2}\rbr^{-\frac{n+2}{p}}2^\frac{n+2}{p}\frac{1}{\bH_1}\la.
\end{equation*}
Therefore taking $\bH_1=\bH_1(\sig,n,p)$, we may assume the following : There exists $j_0\ge1$ such that 
\begin{gather}\label{km}
G_{j_0}\le\frac{1}{64}\la\txt{ and }G_j>\frac{1}{64}\la\text{ for }j>j_0.
\end{gather}
If \cref{km} does not hold, then this trivially implies $\lim_{j \rightarrow \infty} G_j \leq \tfrac{1}{64}\la$ which proves \cref{thm_3} after making use of \cref{def7.4}.

In this subsection, we will prove the following proposition.
\begin{proposition}\label{intr_infty}
	Let $u$ be a weak solution of \cref{p_main}. Suppose \cref{p_la} holds. There exist $\bH_1\ge 1$ and $\bH_2\ge1$ depending on $n$ and $p$ such that if \cref{p_la} holds for some $\la\ge1$, then
	\begin{equation*}
	|\nabla u(z_0)| \apprle \la.
	\end{equation*}
\end{proposition}
The proof follows by induction argument applied to \cref{claim8}. Thus let us prove the following claim:

		\begin{claim}\label{claim8}
 Suppose there exists $j_1 > j_0$ such that  for all $j_0 \leq j \leq j_1$, the following is satisfied:
		\begin{equation}\label{ind}
		\lbr\fiint_{Q_j}|\na u|^p\ dz\rbr^\frac{1}{p}+\eduj\le\la.
		\end{equation}
		Then, the following holds for $j_1+1$:
		\begin{equation}\label{ind_c}
		\lbr\fiint_{Q_{j_1+1}}|\na u|^p\ dz\rbr^\frac{1}{p}+E_p(\na u,Q_{j_1+1})\le\la.\end{equation}
		
		\end{claim}
		
		\begin{proof}[Proof of \cref{claim8}]
			The proof proceeds in two steps.
			\begin{description}[leftmargin=0pt]
		\item [\textcolor{cyan}{STEP 1:}] Let $j_0\le j\le j_1$ and $v_j\in C(I_{\rho_j}^\la(t_0);L^2(B_{\rho_j}(x_0)))\cap L^p(I_{\rho_j}^\la(t_0);W^{1,p}(B_{\rho_j}(x_0)))$ be the weak solution of
		\begin{equation*}
		\begin{cases}
		\partial_tv_j-\dv|\na v_j|^{p-2}\na v_j=0&\text{ in }Q_j,\\
		v_j=u&\text{ on }\partial_pQ_j.
		\end{cases}
		\end{equation*}
		In this step, we want to apply \cref{km_alt} which requires verifying the hypothesis in  \cref{Dib_alt}.

		Applying \cref{diff} when $p \geq 2$ or \cref{p_le_2} when $ 1<p\leq 2$ and \cref{diff_n} along with \cref{p_la} and \cref{p_sum}, we deduce that
		\begin{equation}\label{ind_diff}
			\begin{array}{rcl}
			\lbr\fiint_{Q_j}|\na u-\na v_j|^p\ dz\rbr^\frac{1}{p}\apprle_{(n,p)}\la^{\max\{0,2-p\}}\bF_j&\le& \la^{\max\{0,2-p\}}\sum_{1\le j<\infty} \bF_j \\
			&\overset{\cref{p_la}}{\le}&\mathbf{C}\frac{\la}{\bH_2^{\min\{p-1,1\}}}.
			\end{array}
		\end{equation}
		
		Noting that $\bH_2 \geq 1$, combining \cref{ind} and \cref{ind_diff} gives
		\begin{equation*}
		\lbr\fiint_{Q_j}|\na v_j|^p\ dz\rbr^\frac{1}{p}\le\lbr\fiint_{Q_j}|\na u-\na v_j|^p\ dz\rbr^\frac{1}{p}+\lbr\fiint_{Q_j}|\na u|^p\ dz\rbr^\frac{1}{p}\apprle_{(n,p)}\la,
		\end{equation*}
		and thus \cref{hom_up} implies that there exists $A=A(n,p)$ and $\alpha=\alpha(n,p)\in(0,1)$ such that for any $s\in(0,1)$
		\begin{gather}\label{j_osc}
		\sup_{\frac{1}{2}Q_j}|\na v_j|\le A\la\txt{and}\osc\limits_{\frac{s}{2}Q_{j}}\na v_j\le As^{\alpha}\la.
		\end{gather}
		This proves the upper bound of the hypothesis \cref{Dib_alt} holds.

		We now prove the lower bound of the hypothesis \cref{Dib_alt} holds.  Let $1<k\in\NN$ to be determined later, then the following sequence of estimates hold:
		\begin{equation}\label{7.30eq}
		\begin{array}{rcl}
		\frac{1}{64}\la
		\overset{\cref{km}}{<} G_{j+k}
		&\overset{\redlabel{ind1_a}{a}}{\le}&\lbr\fiint_{Q_{j+k}}|\na v_j|^p\ dz\rbr^\frac{1}{p}+\lbr\fiint_{Q_{j+k}}|\na u-\na v_j|^p\ dz\rbr^\frac{1}{p}\\
		&&+2\lbr\frac{\sig}{2}\rbr^{-\frac{n+2}{p}}\lbr\fiint_{Q_{j+k}}|\na u-(\na v_j)_{Q_{j+k}}|^p\ dz\rbr^\frac{1}{p}\\
		&\overset{\redlabel{ind1_b}{b}}{\le}&\lbr\fiint_{Q_{j+k}}|\na v_j|^p\ dz\rbr^\frac{1}{p}+3\lbr\frac{\sig}{2}\rbr^{-\frac{n+2}{p}}\lbr\fiint_{Q_{j+k}}|\na u-\na v_j|^p\ dz\rbr^\frac{1}{p}\\
		&&\qquad\qquad\qquad+2\lbr\frac{\sig}{2}\rbr^{-\frac{n+2}{p}}\lbr\fiint_{Q_{j+k}}|\na v_j-(\na v_j)_{Q_{j+k}}|^p\ dz\rbr^\frac{1}{p},
		\end{array}
		\end{equation}
		where to obtain \redref{ind1_a}{a} and  \redref{ind1_b}{b}, we used triangle inequality for $\|\cdot\|_p$.
Note that the following bounds hold:
		\begin{equation}\label{7.31eq}
		2\lbr\frac{\sig}{2}\rbr^{-\frac{n+2}{p}}\lbr\fiint_{Q_{j+k}}|\na v_j-(\na v_j)_{Q_{j+k}}|^p\ dz\rbr^\frac{1}{p}\overset{\cref{j_osc}}{\le} 2A\lbr\frac{\sig}{2}\rbr^{-\frac{n+2}{p}+\alpha(k-1)}\la\le 2A\lbr\frac{1}{2}\rbr^{-\frac{n+2}{p}+\alpha(k-1)}\la\le \frac{1}{256}\la,
		\end{equation}
		and
		\begin{equation}
		\label{fix_h_2}
		3\lbr\frac{\sig}{2}\rbr^{-\frac{n+2}{p}}\lbr\fiint_{Q_{j+k}}|\na u-\na v_j|^p\ dz\rbr^\frac{1}{p}\overset{\cref{ind_diff}}{\le} 3\lbr\frac{\sig}{2}\rbr^{-\frac{n+2}{p}(k+1)}\mathbf{C}\frac{\la}{\bH_2^{\min\{p-1,1\}}}\la\le \frac{1}{256}\la,
		\end{equation}
		provided $k=k(n,p)>1+\frac{n+2}{\alpha p}$ and $\bH_2(\sig,n,p)$ are large enough.

		Combining \cref{7.31eq} and \cref{fix_h_2} into \cref{7.30eq}, we have 
		\begin{equation*}
		\frac{1}{128}\la\le\lbr\fiint_{Q_{j+k}}|\na v_j|^p\ dz\rbr^\frac{1}{p}\le \sup_{Q_{j+k}}|\na v_j|\le \sup_{Q_{j+1}}|\na v_j|.
		\end{equation*}

		This shows that \cref{Dib_alt} is satisfied, which says that  for any $\gamma\in(0,1)$ and $\sig=\sig(\gamma,n,p)$, the following conclusion holds:
		\begin{equation}\label{hom_decay}
		E_p(\na v_j,Q_{j+1})\le \gamma E_p(\na v_j,\tfrac{1}{2} Q_j).
		\end{equation}
		
		\item [\textcolor{cyan}{STEP 2:}] In this step, we will prove that the induction conclusion \cref{ind_c} holds. In order to do this,  let us first  estimate the decay of $\eduj$ as follows:
		\begin{equation*}
		\begin{array}{rcl}
		\edujj
		&\le&2\lbr\fiint_{Q_{j+1}}|\na u-(\na v_j)_{Q_{j+1}}|^p\ dz\rbr^\frac{1}{p}\\
		&\le&2\lbr\fiint_{Q_{j+1}}|\na u-\na v_j|^p\ dz\rbr^\frac{1}{p}+2\edvjj\\
		&\overset{\cref{hom_decay}}{\le}&2\lbr\frac{\sig}{2}\rbr^{-\frac{n+2}{p}}\lbr\fiint_{Q_{j}}|\na u-\na v_j|^p\ dz\rbr^{\frac{1}{p}}+2\gamma\edvj.
		\end{array}
		\end{equation*}
		Similarly, there holds
		\begin{equation*}
		\begin{array}{rcl}
		\edvj
		&\le& 2^{1+\frac{n+2}{p}}\lbr\fiint_{Q_{j}}|\na u-\na v_j|^p\ dz\rbr^\frac{1}{p}+2E_p(\na u,\tfrac{1}{2}Q_{j})\\
		&\le&2^{1+\frac{n+2}{p}}\lbr\fiint_{Q_{j}}|\na u-\na v_j|^p\ dz\rbr^\frac{1}{p}+2^{2+\frac{n+2}{p}}E_p(\na u,Q_{j}).
		\end{array}    	   	
		\end{equation*}
		Thus, using \cref{ind_diff}, we get
		\begin{equation}\label{ppre_decay}
		\edujj\le  C(n,p)\lbr\frac{\sig}{2}\rbr^{-\frac{n+2}{p}} \la^{\max\{0,2-p\}}\bF_j+2^{3+\frac{n+2}{p}}\gamma\eduj.
		\end{equation}
		
		We take $\gamma=\frac{1}{2^{4+\frac{n+2}{p}}}$ which in turn fixes  $\sig=\sig(n,p)$ according to \cref{hom_decay} so that \cref{ppre_decay} becomes
		\begin{equation}\label{pre_decay}
		\edujj\le \frac{1}{2}\eduj
		+ \la^{\max\{0,2-p\}}\bF_j.
		\end{equation}
		It follows that
		\begin{equation}\label{decay}
		\begin{array}{rcl}
		\sum_{j=j_0}^{j_1}\edujj
		&\le& E_p(\na u,Q_{j_0})+C\sum_{j=j_0}^{j_1} \la^{\max\{0,2-p\}}\bF_j\\
		&\overset{\cref{km},\cref{ind_diff}}{\le}&\frac{1}{64}\lbr\frac{\sig}{2}\rbr^{\frac{n+2}{p}}\la+\mathbf{C}\frac{\la}{\bH_2^{\min\{p-1,1\}}}\\
		&\le& \frac{1}{32}\lbr\frac{\sig}{2}\rbr^{\frac{n+2}{p}}\la,
		\end{array}
		\end{equation}
		provided $\bH_2=\bH_2(n,p)$ is large enough.	Since there holds
		\begin{equation*}
		|(\na u)_{Q_{j+1}}-(\na u)_{Q_{j}}|\le \fiint_{Q_{j+1}}|\na u-(\na u)_{Q_j}|\ dz\le \lbr\fiint_{Q_{j+1}}|\na u-(\na u)_{Q_j}|^p\ dz\rbr^\frac{1}{p}\le\lbr\frac{\sig}{2}\rbr^{-\frac{n+2}{p}}\eduj,
		\end{equation*}
		we have
		\begin{equation}\label{pre_end}
		|(\na u)_{Q_{j_1+1}}|\le|(\na u)_{Q_{j_0}}|+\lbr\frac{\sig}{2}\rbr^{-\frac{n+2}{p}}\sum_{j=j_0}^{j_1}\eduj \overset{\cref{decay}}{\le}\lbr\fiint_{Q_{j_0}}|\na u|^p\ dz\rbr^\frac{1}{p}+\frac{1}{32}\la\overset{\cref{km}}{\le}\frac{1}{16}\la.
		\end{equation}
		Hence, combining \cref{decay} and \cref{pre_end} along with the fact that $\sigma \in (0,\tfrac14)$, we obtain
		\begin{equation*}
		\lbr\fiint_{Q_{j_1+1}}|\na u|^p\ dz\rbr^\frac{1}{p}+E_p(\na u,Q_{j_1+1})\le |(\na u)_{Q_{j_1+1}}|+2E_p(\na u,Q_{j_1+1})\le\la.
		\end{equation*}
	\end{description}
	This completes the proof of claim.
\end{proof}

\subsection{Proof of \texorpdfstring{\cref{thm_3}}.}\label{final_pf}

The proof of \cref{intr_infty} requires the hypothesis \cref{p_la} hold, which we will prove in this subsection, thus completing the proof of the theorem. Let $\bH_1\ge1$ and $\bH_2\ge1$ be fixed as in \cref{intr_infty} and we show the existence of $\la\geq 1$ satisfying \cref{p_la} as follows:
\begin{description}[leftmargin=12pt]
	\item [Case $p\ge2$: ] For $\mu\ge1$, let us define 
	\begin{equation*}
	h(\mu):=\mu^p-\lbr\bH_1^p\fiint_{Q_{2\rho}^\mu}|\na u|^p+1\ dz+\bH_2^p\bF^{p}\rbr.
	\end{equation*}
	Then we see that $h(\mu)$ is a continuous function satisfying  $h(1)<0$ and $\lim_{\mu\to\infty}h(\mu)=\infty$ because there holds
	\begin{equation*}
	\begin{array}{rcl}
	h(\mu)
	&=&\mu^p-\lbr\mu^{p-2}\bH_1^p\frac{1}{|Q_{2\rho}|}\iint_{Q_{2\rho}^\mu}|\na u|^p+1\ dz+\bH_2^p\bF^{p}\rbr\\
	&\ge&\mu^p-\lbr\mu^{p-2}\bH_1^p\fiint_{Q_{2\rho}}|\na u|^p+1\ dz+\bH_2^p\bF^{p}\rbr.
	\end{array}
	\end{equation*}
	Thus there exists $\la>1$ such that $h(\la)=0$ or in particular, the hypothesis \cref{p_la} holds. 
	Also, $h(\la)=0$ implies $\la^p\ge\bH_2^p\bF^p$ and
	\begin{equation*}
	\begin{array}{rcl}
	\la^p
	=\bH_1^p\fiint_{Q_{2\rho}^\la}|\na u|^p+1\ dz+\bH_2^p\bF^p\le\la^{p-2}\bH_1^p\fiint_{Q_{2\rho}}|\na u|^p+1\ dz+\la^{p-1}\bH_2\bF,
	\end{array}
	\end{equation*}
	which gives
	\begin{equation*}
	\la^2\le\bH_1^p\fiint_{Q_{2\rho}}|\na u|^p+1\ dz+\la\bH_2\bF\le \bH_1^p\fiint_{Q_{2\rho}}|\na u|^p+1\ dz+\frac{1}{2}\la^2+\frac{1}{2}\bH_2^2\bF^2.
	\end{equation*}
	Thus, we obtain 
	\begin{equation}\label{stop_1}
	\begin{array}{rcl}
	\la\apprle \lbr\fiint_{Q_{2\rho}}|\na u|^p+1\ dz\rbr^\frac{1}{2}+\bF.
	\end{array}
	\end{equation}
	Substituting \cref{stop_1} into \cref{intr_infty} gives
	\begin{equation*}
	|\nabla u(z_0)| \le \la\apprle_{(n,p)}\lbr\fiint_{Q_{2\rho}}|\na u|^p+1\ dz\rbr^\frac{1}{2}+\bF.
	\end{equation*}

	\item [Case $\frac{2n}{n+2}<p<2$:] Similarly, for $\mu\ge1$ and $\rho^\mu:=\mu^{\frac{p-2}{2}}\rho$, we define
	\begin{equation*}
	h(\mu):=\mu^p-\lbr\bH_1^p\fiint_{Q_{2\rho^\mu}^\mu}|\na u|^p+1\ dz+\bH_2^p\bF^\frac{p}{p-1}\rbr.
	\end{equation*}
	Analogous to the degenerate case, we can find  $\la>\max\{1,\bH_2\bF^\frac{1}{p-1}\}$ such that $h(\la)=0$. 
	Since there holds that
	\begin{equation*}
	0<\frac{p}{d}=\frac{p(n+2)-2n}{2}=p-\frac{(2-p)n}{2},
	\end{equation*}
	we see that
	\begin{equation*}
	\la^p\leq   \la^{p-\frac{p}{d}}\bH_1^p\fiint_{Q_{2\rho}}|\na u|^p+1\ dz+\la^{p-\frac{p}{d}}\lbr\bH_2\bF^\frac{1}{p-1}\rbr^\frac{p}{d}.
	\end{equation*}
	Thus, we get
	\begin{equation*}
	\la^{\frac{p}{d}}\le\bH_1^p\fiint_{Q_{2\rho}}|\na u|^p+1\ dz+\lbr\bH_2\bF^\frac{1}{p-1}\rbr^\frac{p}{d}.
	\end{equation*}
	Combining with \cref{intr_infty}, we obtain 
	\begin{equation*}
	|\nabla u(z_0)| \apprle_{(n,p)}\lbr\fiint_{Q_{2\rho}}|\na u|^p+1\ dz\rbr^\frac{d}{p}+\bF^\frac{1}{p-1}.
	\end{equation*}
\end{description}
\begin{remark}
	We remark that \cref{p_la} is equivalent to 
	\begin{equation*}
	\bH_1^p\fiint_{Q_{2\rho}^\la}|\na u|^p+1\ dz+\la^{\max\{1,p-1\}}\bH_2\mathbf{F}=\la^p.
	\end{equation*}
	Applying Young's inequality, one can recover \cref{p_cancel} and  \cref{p_le_2}. The Proof of \cref{thm_3} also follows analogously.
\end{remark}

\section*{Data availability statement}
Data sharing not applicable to this article as no datasets were generated or analysed during the current study.

\section*{References}

\begin{thebibliography}{10}
	
	\bibitem{MR0450957}
	Robert~A. Adams.
	\newblock {\em Sobolev spaces}.
	\newblock Pure and Applied Mathematics, Vol. 65. Academic Press [Harcourt Brace
	Jovanovich, Publishers], New York-London, 1975.
	
	\bibitem{MR4078712}
	Lisa Beck and Giuseppe Mingione.
	\newblock Lipschitz bound and nonuniform ellipticity.
	\newblock {\rm Comm. Pure Appl. Math.}, 73(5):944--1034, 2020.
	
	\bibitem{MR2122416}
	Bildhauer, Michael and Fuchs, Martin and Zhong, Xiao
	\newblock A lemma on the higher integrability of functions with
	applications to the regularity theory of two-dimensional
	generalized {N}ewtonian fluids.
	\newblock {\rm Manuscripta Math.}, 116(2):135--156, 2005.
	
	\bibitem{MR4377996}
	Verena B\"{o}gelein, Frank Duzaar, Paolo Marcellini, and Christoph Scheven.
	\newblock Boundary regularity for elliptic systems with {$p,q$}-growth.
	\newblock {\em J. Math. Pures Appl. (9)}, 159:250--293, 2022.
	
	\bibitem{MR4021174}
	Sun-Sig Byun and Yeonghun Youn.
	\newblock Potential estimates for elliptic systems with subquadratic growth.
	\newblock {\em J. Math. Pures Appl. (9)}, 131:193--224, 2019.

	
	\bibitem{MR4438896}
	Cristiana De~Filippis.
	\newblock Quasiconvexity and partial regularity via nonlinear potentials.
	\newblock {\em J. Math. Pures Appl. (9)}, 163:11--82, 2022.
	
	\bibitem{MR4331020}
	Cristiana De~Filippis and Giuseppe Mingione.
	\newblock Lipschitz bounds and nonautonomous integrals.
	\newblock {\em Arch. Ration. Mech. Anal.}, 242(2):973--1057, 2021.
	
	
	\bibitem{MR1230384}
	Emmanuele DiBenedetto.
	\newblock {\em Degenerate parabolic equations}.
	\newblock Universitext. Springer-Verlag, New York, 1993.
	
	\bibitem{MR2729305}
	Frank Duzaar and Giuseppe Mingione.
	\newblock Gradient continuity estimates.
	\newblock {\em Calc. Var. Partial Differential Equations}, 39(3-4):379--418,
	2010.
	
	\bibitem{MR2719282}
	Frank Duzaar and Giuseppe Mingione.
	\newblock Gradient estimates via linear and nonlinear potentials.
	\newblock {\em J. Funct. Anal.}, 259(11):2961--2998, 2010.
	
	\bibitem{MR1962933}
	Enrico Giusti.
	\newblock {\em Direct methods in the calculus of variations}.
	\newblock World Scientific Publishing Co., Inc., River Edge, NJ, 2003.
	
	\bibitem{MR1730563}
	Tadeusz Iwaniec and Anne Verde.
	\newblock On the operator {${\mathcal L}(f)=f\log|f|$}.
	\newblock {\em J. Funct. Anal.}, 169(2):391--420, 1999.
	
	\bibitem{MR1749438}
	Juha Kinnunen and John~L. Lewis.
	\newblock Higher integrability for parabolic systems of {$p$}-{L}aplacian type.
	\newblock {\em Duke Math. J.}, 102(2):253--271, 2000.
	
	\bibitem{MR2968162}
	Tuomo Kuusi and Giuseppe Mingione.
	\newblock New perturbation methods for nonlinear parabolic problems.
	\newblock {\em J. Math. Pures Appl. (9)}, 98(4):390--427, 2012.
	
	\bibitem{MR3004772}
	Tuomo Kuusi and Giuseppe Mingione.
	\newblock Linear potentials in nonlinear potential theory.
	\newblock {\em Arch. Ration. Mech. Anal.}, 207(1):215--246, 2013.
	
%
	\bibitem{MR3273649}
	Tuomo Kuusi and Giuseppe Mingione.
	\newblock Borderline gradient continuity for nonlinear parabolic systems.
	\newblock {\em Math. Ann.}, 360(3-4):937--993, 2014.
	
	\bibitem{MR3174278}
	Tuomo Kuusi and Giuseppe Mingione.
	\newblock Guide to nonlinear potential estimates.
	\newblock {\em Bull. Math. Sci.}, 4(1):1--82, 2014.
	
	\bibitem{MR3247381}
	Tuomo Kuusi and Giuseppe Mingione.
	\newblock A nonlinear {S}tein theorem.
	\newblock {\em Calc. Var. Partial Differential Equations}, 51(1-2):45--86,
	2014.
	
	\bibitem{MR1461542}
	Jan Mal\'{y} and William~P. Ziemer.
	\newblock {\em Fine regularity of solutions of elliptic partial differential
		equations}, volume~51 of {\em Mathematical Surveys and Monographs}.
	\newblock American Mathematical Society, Providence, RI, 1997.
	
	\bibitem{MR2746772}
	Giuseppe Mingione.
	\newblock Gradient potential estimates.
	\newblock {\em J. Eur. Math. Soc. (JEMS)}, 13(2):459--486, 2011.
	
	\bibitem{MR2342615}
	Mikko Parviainen.
	\newblock Global higher integrability for parabolic quasiminimizers in
	nonsmooth domains.
	\newblock {\em Calc. Var. Partial Differential Equations}, 31(1):75--98, 2008.
	
	\bibitem{MR2491806}
	Mikko Parviainen.
	\newblock Global gradient estimates for degenerate parabolic equations in
	nonsmooth domains.
	\newblock {\em Ann. Mat. Pura Appl. (4)}, 188(2):333--358, 2009.
	
	\bibitem{MR2468726}
	Mikko Parviainen.
	\newblock Reverse {H}\"{o}lder inequalities for singular parabolic equations
	near the boundary.
	\newblock {\em J. Differential Equations}, 246(2):512--540, 2009.
	
\end{thebibliography}

\end{document}